\documentclass{amsart}
\usepackage[utf8]{inputenc}
\usepackage{amsmath,amsthm}
\usepackage{epsfig,psfrag,graphicx,sidecap}

\usepackage{latexsym,lmodern,amssymb}
\usepackage{dsfont,txfonts}
\usepackage{amsfonts,enumerate}
\usepackage{variations}
\usepackage{fontenc}

\newtheorem{theorem}{Theorem}[section]

\newtheorem{corollary}[theorem]{Corollary}
\newtheorem{definition}[theorem]{Definition}
\newtheorem{proposition}[theorem]{Proposition}
\newtheorem{lemma}[theorem]{Lemma}

\theoremstyle{definition}
\newtheorem{remark}[theorem]{Remark}

\numberwithin{equation}{section}

\begin{document}

\title[Quantitative isoperimetric inequalities]{Quantitative isoperimetric inequalities for log-convex probability measures on the line}

\author{F. Feo, M.R. Posteraro, C. Roberto}

\date{\today}

\address{Dipartimento di Ingegneria, Universit\`a degli Studi di Napoli Parthenope,
Centro Direzionale Isola C4, 80100 Naples, Italy}

\address{Dipartimento di Matematica e Applicazioni "Renato Caccioppoli", Via Cintia - Complesso Monte  S.Angelo, 80100 Naples, Italy}

\address{Universit\'e Paris Ouest Nanterre la D\'efense, MODAL'X, EA 3454, 200 avenue de la R\'epublique 92000 Nanterre, France}

\email{filomena.feo@uniparthenope.it, posterar@unina.it,\newline croberto@math.cnrs.fr}



\begin{abstract}
The purpose of this paper is to analyze the isoperimetric inequality for symmetric log-convex probability measures on the line. Using geometric arguments we first re-prove that extremal sets in the isoperimetric inequality are
intervals or complement of intervals (a result due to Bobkov and Houdr\'e). Then we
give a quantitative form of the isoperimetric inequality, leading to a somehow anomalous behavior.
Indeed, it could be that a set is very close to be optimal, in the sense that the isoperimetric inequality is almost an equality, but at the same time is very far (in the sense of the symmetric difference between sets) to any extremal sets!
From the results on sets we derive quantitative functional inequalities of
weak Cheeger type.
\end{abstract}

\maketitle


\section{Introduction}

The isoperimetric problem in probability spaces is a very rich and extensive theory, with many applications in probability, analysis and geometry, as for example concentration of measure, phenomena in high dimension \cite{ledoux-99}, rearrangement, PDE's \cite{martin-milman-11} etc. See  \textit{e.g.}\ \cite{ledoux-saintflour,barthe-02,roberto-10,milman} for overview papers and monographs.

 The isoperimetric inequality for the Gaussian measure in dimension 1 (the result hold in any dimension \cite{borell75,sudakov-tsirelson}) reads
$$
P(E) \geq I(\gamma(E)) \qquad \mbox{for all  Borel set } E \subset \mathbb{R},
$$
 where $P(E)$ is the perimeter of $E$ (see below for a precise definition), $\gamma(E)=\frac{1}{\sqrt{2\pi}}
 \int_E e^{-x^2/2}dx$ is the measure of the set $E$ with respect to the Gaussian measure $\gamma$
 and $I=\varphi \circ \Phi^{-1}$ is the \emph{isoperimetric profile} (here $\varphi$ stands for the density of $\gamma$ and $\Phi$ for its cumulative  distribution function). Equality cases are given by half-lines (half spaces in dimension higher than 1  \cite{ehrhard-86,carlen-kerce}). Very recently Cianchi, Fusco, Maggi and Pratelli \cite{fusco} solved the harder question about the almost equality cases (see \cite{mossel-neeman1,mossel-neeman2,Eldan} for further developments). If $\bar E$ is an extremal set in the above isoperimetric inequality, defining \emph{the deficit} as
$$
 \delta(E) := P(E) - P(\bar E),
$$
the authors proved the following quantitative isoperimetric inequality
\begin{equation}\label{deficit-intro}
\delta(E) \geq C(\gamma(E))\lambda(E)\sqrt{\log\frac{1}{\lambda(E)}},
\end{equation}
where $\lambda(E):=\inf_{\genfrac{}{}{0pt}{}{H \textrm{half line:}}{\gamma(H)=\gamma(E)}} \gamma(E \Delta H)$, $\Delta$ stands for the symmetric difference between sets and $C$ is a constant that depends on
the measure $\gamma(E)$ of the set.
The quantity $\lambda(E)$ is called the \emph{asymmetry} of $E$: it encodes, in the sense of the symmetric difference, how far the set is from the extremal sets in the isoperimetric inequality.
One of the main issue here is to find the sharp dependence in $\delta$.
Observe that \eqref{deficit-intro} relates two different "distances" from a set $E$ to the extremal sets in the isoperimetric inequality.

Using a geometric argument in the spirit of \cite{fusco} de Castro in \cite{decastro} is able to identify all extremal sets in the isoperimetric inequality that have also a fixed asymmetry. More precisely, he proves that  among sets of given measure \emph{and} given asymmetry, intervals or complements of intervals (depending on the range) have minimal perimeter. Furthermore, he deals more generally with any log-concave probability measure and not only the Gaussian measure.

\medskip

In the present paper, our aim is to analyze quantitative isoperimetric inequalities for the class of log-convex probability measures on the line. Assume for simplicity that $\mu$ is a symmetric probability measure on the line, with density $f$ (no atoms). Then $\mu$ is said to be \emph{log-convex} if $\log f$ is convex on $(-\infty,0]$. This class of probability measures includes for example the generalized Cauchy distributions
$dm_{\alpha}(x)=\frac{\alpha dx}{2\left(  1+\left\vert x\right\vert \right)^{1+\alpha}}$, $\alpha >0$, (such distributions play an important role in probability and analysis: they are related to the well-known $\kappa$-concave probability measures \cite{borel1}, are Barentblatt solution of Porous Medium Equation \cite{barenblatt} and are extremal functions in the classical Sobolev inequality \cite{aubin,talenti-76}).
The isoperimetric problem for log-convex probability measures  in dimension 1 is fully solved
by Bobkov and Houdr\'e  \cite{bob-hudre}(see also \cite {morgan}). In higher dimension, for product of log-convex probability measures, extremal sets are not known. However, some estimates on the isoperimetric profile
(dimension dependent as it must be \cite{talagrand-91}) are given in \cite{roberto} (see also \cite{rockner-wang,bcr-05}) with links with the concentration of measure.

Using a geometric argument of the type of Cianchi et al. \cite{fusco}, we shall first re-prove
the result by Bobkov and Houdr\'e \cite{bob-hudre} on the extremal sets in the isoperimetric inequality (see Section 3).
Then, we will obtain a quantitative isoperimetric inequality (in the form of \eqref{deficit-intro}, see Section 4) which appears to be surprising, due to the presence of different shapes in the extremal sets when the measure of the set is precisely $1/2$. Indeed, it could be that a set has very small deficit ($\delta$ above) but large asymmetry (see Section 4.2). This is one of our main result. We emphasize that log-convex probability measures are the first examples  of measures, to the best of our knowledge, displaying such an anomalous property.

Contrary to the case of the log-concave probability measures, there is not a unique description of extremal sets with given measure and given asymmetry. We shall illustrate this with two explicit examples (see Section \ref{sec:preliminary}). However, under few additional assumptions on the density $f$ of the measure, we will give a  unified description of extremal sets with given measure and given asymmetry. From our estimates on sets we finally derive quantitative functional inequalities of weak Cheeger type in some specific cases (see Section \ref{sec:functional}).
\medskip

There is an important activity on the questions of quantitative inequalities. To give a complete overview of the literature would be out of reach. Let us mention only few very recent works somehow related to the present paper.
In \cite{figali-maggi}
the authors deal with the isoperimetric problem for radially symmetric log-convex probability measures (finding extremal sets in $\mathbb{R}^n$).
In \cite{figali-m-p-09} the authors deal with quantitative Brunn-Minkowski inequality (which is related to
the isoperimetric problem in Euclidean space), while functional counter parts can be found in
\cite{fusco-sobolev,dolbeault} on Sobolev inequalities, and in \cite{indrei,bgrs} on log-Sobolev inequalities for the Gaussian measure.



\section{Log-convex measures: definition and first properties} \label{sec:log-convex}

In this section we introduce the notion of log-convex probability measures on the line,  we give some examples and prove few basic properties.

Throughout the paper, we assume that $\mu$ is a probability measure on $\mathbb{R}$
 with no atoms and with density $f$. Set $F(x)=\mu\left(  \left(  -\infty
,x\right]  \right)$,  $x \in \mathbb{R}$, for its distribution function
and let
\[
a=\inf\left\{  x\in \mathbb{R}
:F(x)>0\right\}
\qquad \text{ and } \qquad
b=\sup\left\{  x\in
\mathbb{R}
:F(x)<1\right\}  .
\]
In general $-\infty\leq a<b\leq+\infty.$
In analogy with the family of log-concave probability measures, we define the family of log-convex probability measures.

\begin{definition}[Log-convex measure]\label{def log-convex}
Let $\mu$ be a probability measure,  on $\mathbb{R}$, with no atoms and density $f$.
Then, $\mu$ is said to be \emph{log-convex} (respectively \emph{strictly log-convex})
if there exists $x_{0} \in \left(  a,b\right)$  such that $\log f$ is convex (respectively strictly convex) on
$\left(a,x_{0}\right)$ and on $\left(  x_{0},b\right)$.
\end{definition}

The family of log-convex measures includes the generalized Cauchy distributions
\begin{equation} \label{mis 1}%
dm_{\alpha}(x)=\frac{\alpha}{2\left(  1+\left\vert x\right\vert \right)
^{1+\alpha}}dx
\end{equation}
where $\alpha >0$ is a parameter (notice that Cauchy distribution are strictly log-convex).
It also includes the two-sided exponential measures
\begin{equation} \label{misura 0}
d\mu_{1}(x)=\frac{e^{-\left\vert x\right\vert }}{2}dx
\end{equation}
(which is not strictly log-convex) and more generally any probability measures of the form
\begin{equation} \label{mis mu fi}%
d\mu_{\Phi}(x)=Z_{\Phi}^{-1}e^{-\Phi(x)}dx
\end{equation}
 with $\Phi$ \textit{e.g.}\ even and concave on $\left(  0,+\infty\right)$.
The family of log-convex measures intersects (but is different from) the family of $\kappa$-concave probability measures introduced by Borell \cite{borel1}.

Observe that
$F$ is strictly increasing on $\left(  a,b\right)$
and if one sets
\begin{equation} \label{J}
J(t)=f\left(  F^{-1}(t)\right)  \text{ \ \ \ \ \ }0<t<1 ,
\end{equation}
then $\underset{t\rightarrow0}{\lim}J(t)=\underset{t\rightarrow1}{\lim}J(t)=0$, and
the map $t \mapsto J(t)$ is increasing on $\left(  0,F^{-1}(x_{0})\right)  $ and decreasing on $\left(  F^{-1}(x_{0}),1\right)$.
Without any further mention, in the rest of the paper we will extend $J$ up to $0$ and $1$, setting $J(0)=J(1)=0$.
For example, for the two-sided exponential measure \eqref{misura 0}, one can easily check that
$J(t)=\min(t,1-t)$, that for the generalized Cauchy distributions \eqref{mis 1},
$J_\alpha(t)=\alpha 2^\frac{1}{\alpha}\min(t,1-t)^{1+\frac{1}{\alpha}}$ (see e.g. \cite{roberto}) and that for
\eqref{mis mu fi}, under mild assumption on $\Phi$ (see \cite[Proposition 5.21]{roberto} for a precise statement),
$J_\Phi(t) \sim t\Phi^{\prime}\left( \Phi^{-1}\left(  \log\frac{1}{t}\right)\right)$, as $t$ goes to 0.

The following characterization holds.

\begin{proposition}\label{prop mu-J}
Let $\mu$ be a non-atomic probability measure with density $f$ and distribution
function $F$. Let $a=\inf\left\{  x\in \mathbb{R} :F(x)>0\right\}$
 and $b=\sup\left\{  x\in \mathbb{R} :F(x)<1\right\}$. Assume that $F$ is strictly increasing on $\left(a,b\right)$ and denote by $F^{-1} \colon \left( 0,1\right)  \rightarrow \left(a,b\right)$ the inverse of $F$.
Then, the following properties are equivalent:
\begin{enumerate}[(i)]
\item $\mu$ is log-convex (resp. strictly log-convex);
\item $f$ is continuous and positive on $\left(  a,b\right)$ and $J=f \circ F^{-1}$
is convex (resp. strictly convex) on $\left(0,F^{-1}\left(  x_{0}\right)  \right)$ and on $\left(F^{-1}\left(  x_{0}\right)  ,1\right)$.
\end{enumerate}
\end{proposition}

The proof (that we omit) is analogous to the case of log-concave measures (see
\cite[Proposition A.1]{bobkov}).

For simplicity (mainly to avoid unnecessary technicalities), we will restrict ourself to the study of a sub-class of log-convex probability measures.

In what follows, we will only consider log-convex probability measures symmetric
with respect to the origin \footnote{Observe that the choice of the origin is not
restrictive since the measure $\mu(\cdot+\alpha)$, with $\alpha\in\mathbb{R}$, shares the same isoperimetric properties
as the measure $\mu$.} (i.e. $x_{0}=0$ and $J$ is symmetric with respect to $\frac{1}{2}$), such that
$a=\inf\left\{  x\in \mathbb{R} :F(x)>0\right\}=-\infty$ and $b=\sup\left\{  x\in \mathbb{R} :F(x)<1\right\}=+\infty$.
\begin{definition}[the set $\mathcal{F}$]\label{def:f}
We set $\mathcal{F}$ for the set of all strictly log-convex probability measures $\mu$ on $\mathbb{R}$,
symmetric with respect to the origin, satisfying $$\inf\left\{  x\in \mathbb{R} :F(x)>0\right\}=-\infty \text{ and }
\sup\left\{  x\in \mathbb{R} :F(x)<1\right\}=+\infty.$$
\end{definition}

The following Lemma holds

\begin{lemma} \label{lemino}
Let $\mu\in\mathcal{F}$
. Then $t \mapsto \frac{J(t)}{t}$ is a
strictly increasing
 function on $[0,\frac{1}{2}]$.
\end{lemma}




\section{Isoperimetric inequality}

In this section we recover known isoperimetric inequalities for log-convex measures on the line \cite{bob-hudre}, using
 geometric arguments (\cite{fusco,decastro}) that will allow us, in the next section, to prove quantitative estimates.

Given a probability measure $\mu$ on the line with density $f$,
$\mu$-perimeter of Borel set $E$ is defined as
\[
P_\mu(E)=\int_{\partial^{M}E}f(x)d\mathcal{H}^{0}(x),
\]
where $\mathcal{H}^{0}(x)$ denotes the 0-dimensional Hausdorff measure in $\mathbb{R}$ and
$\partial^{M}E$\ is the essential boundary of $E$ (see \textit{e.g.}\ \cite[Page 108-112]{libro-fusco}).
In most occurrences we will write for simplicity $P$ for $P_{\mu}$.

Bobkov and Houdr\'e proved the following very general statement.

\begin{theorem}[\cite{bob-hudre} Corollary 13.10] Let $d\mu_{\Phi}(x)=Z_{\Phi}^{-1}e^{-\Phi(x)}dx$ be a probability measure, with
$\Phi\colon \mathbb{R} \rightarrow \mathbb{R}$ even, and $Z_\Phi$ the normalization constant. Then the extremal sets in the isoperimetric inequality can be found among half-lines, symmetric segments and their complements.
\end{theorem}

Our goal is to state a more precise result about extremal sets for symmetric
strictly log-convex probability measures. As already mentioned, we will make use of a geometric argument that we present now.


\subsection{The shifting property}

Following \cite{fusco} (for Gauss
measures, see also \cite{decastro} for log-concave measure on the line), we prove in this section a
"shifting property" for intervals and complement of intervals.


We start with a definition of shifted intervals.
\begin{definition}[Right/left shifted interval]\label{def shift}
Let $\left(a,b\right)$ be an interval of $\mathbb{R}$ with $-\infty<a<b<+\infty$. Then,
\begin{itemize}
\item any interval $\left(a^{\prime},b^{\prime}\right)$ such that
$a<a^{\prime}<b^{\prime}\leq+\infty$ and $\mu\left(\left(a,b\right)\right)
=\mu\left(  \left(  a^{\prime},b^{\prime}\right)  \right)$ is said to be a \emph{right-shifted interval} of $\left(
a,b\right)$;
\item any interval $\left(a^{\prime},b^{\prime}\right)$ such that
$-\infty \leq a^{\prime} < b^{\prime}<b$ and $\mu\left(\left(a,b\right)\right)
=\mu\left( \left(  a^{\prime},b^{\prime}\right)  \right)$ is said to be a \emph{left-shifted interval} of $\left(
a,b\right)$.
\end{itemize}
\end{definition}

The next proposition is one of our key ingredient. It encodes the fact that, depending on the measure of the interval and on its position on the line, the $\mu$-perimeter decreases for left/right shifted intervals.

\begin{proposition}[Shifting property for intervals] \label{shifting}
Let $\mu \in \mathcal{F}$ (see Definition \ref{def:f}).
\begin{enumerate}[(1)]
\item Let $\left(a,b\right)$ be an interval of measure $\mu\left(\left(  a,b\right)  \right)  <\frac{1}{2}$.
\begin{enumerate}[]
\item{(1a)} If $a\geq0$ (resp. $b\leq0$), then
\[
P\left(\left(  a,b\right)  \right)>P\left(\left(  a^{\prime},b^{\prime}\right)  \right)
\]
for any right-shifted (resp. left-shifted) interval of $\left(a,b\right)$.

\item{(1b)} If $a<0$, $b>0$ and $a+b\geq0$ (resp. $a+b\leq0$), then
\[
P\left(\left(  a,b\right)  \right)>P\left(\left(  a^{\prime},b^{\prime}\right)  \right)
\]
for any left-shifted (resp. right-shifted) interval of $\left(a,b\right)$
with $a^{\prime}+b^{\prime}\geq0$ (resp. $a^{\prime}+b^{\prime}\leq0$).
\end{enumerate}
\item Let $\left(a,b\right)$ be an interval of measure $\mu\left(\left( a,b\right)\right)\geq\frac{1}{2}$.
If $a+b\geq0$ (resp. $a+b\leq0$), then
\[
P\left(\left(  a,b\right)  \right)>P\left(\left(  a^{\prime},b^{\prime}\right)  \right)
\]
for any left-shifted (resp. right-shifted) interval of $\left(a,b\right)$ with
$a^{\prime}+b^{\prime}\geq0$ (resp. $a^{\prime}+b^{\prime}\leq0$).
\end{enumerate}
\end{proposition}

\begin{remark}
Without the strict log-convexity assumption, the results above still hold but no more with strict inequalities.
\end{remark}

\begin{proof}
Let $\left(a,b\right)$ be an interval of measure $p\in\left(0,1\right)$, with $-\infty\leq a<b\leq+\infty$.
Its perimeter is
\[
P\left((a,b)\right)=f(a)+f(b)=J(F(a))+J(F(b)) .
\]
The value $p\in\left(  0,1\right)$ being fixed, necessarily $b=F^{-1}\left(p+F\left(a\right)\right)$. Denoting $a=F^{-1}(t)$, we
may study the function
\begin{equation*} 
P\left((a,b)\right)=J(t)+J(p+t):=\psi_{p}(t)
\end{equation*}
as a function of $t\in[0,1-p]$.
The expected results follow at once from Lemma \ref{lemino2} below.
\end{proof}

\begin{lemma}\label{lemino2}
Let $\mu \in \mathcal{F}$.
For $p\in(0,1)$ and $t \in \left[ 0,1-p\right]$, set $\psi_{p}(t)=J(t)+J(p+t)$. Then,
\begin{enumerate}[i)]
\item $\psi_p$ is symmetric about $\frac{1-p}{2}$,
\item if $p \geq 1/2$, $\psi_p$ is strictly convex on $\left[0,1-p\right]$, decreasing on $\left(0,\frac{1-p}{2}\right)$, increasing on  $\left(\frac{1-p}{2},1-p\right)$, $\frac{1-p}{2}$ is a minimum and $0,1-p$ are maxima.
\item if $p <1/2$, $\psi_{p}$ is strictly convex on $\left(0,\frac{1}{2}-p\right)$,
on $\left(\frac{1}{2}-p,\frac{1}{2}\right)$ and on $\left(\frac{1}{2},1-p\right)$, it is increasing on $\left(0,\frac{1}{2}-p\right)$, decreasing on  $\left(\frac{1}{2}-p,\frac{1-p}{2}\right)$, increasing on $\left(\frac{1-p}{2},\frac{1}{2}\right)$
and decreasing on  $\left(\frac{1}{2},1-p\right)$. Moreover
$\frac{1}{2}-p, \frac{1}{2}$ are maxima and $\frac{1-p}{2}, 0, 1-p$ are minima.
\end{enumerate}
\end{lemma}

\begin{figure}[ht]
\psfrag{p>}{$p \geq \frac{1}{2}$}
\psfrag{p<}{$p<\frac{1}{2}$}
\psfrag{to}{\tiny$\!\!\!\frac{1-p}{2}$}
\psfrag{12}{\tiny$\;\frac{1}{2}$}
\psfrag{1-p}{\tiny$1-p$}
\psfrag{12-p}{\tiny$\!\frac{1}{2}-p$}
\psfrag{t}{$t$}
\begin{center}
\includegraphics[width=.80\columnwidth]{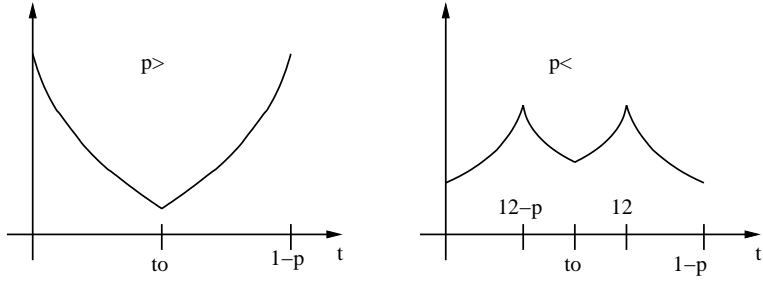}
\end{center}
\caption{The shape of the function $t \mapsto \psi_p(t)$ for $p \geq 1/2$ and $p < 1/2$.}
\label{fig:shiftinglemma}
\end{figure}

\begin{proof}
The proof is elementary and left to the reader.
\end{proof}

We end this section with a converse of Proposition \ref{shifting}, namely that the shifting property of Proposition
\ref{shifting} (to be really precise only a weaker form is needed) implies that $\mu$ is log-convex.

\begin{proposition}
Let $\mu$ be a probability measure on the line, symmetric with respect to a point (say the origin for simplicity) with
density $f$. Assume that $f$ is continuous, positive on $(-\alpha,\alpha)$ for some $\alpha \in (0,\infty]$ and such that
the following shifting property holds:

$1)$ if\textit{ }$\left(  a,b\right)  $\textit{ is an interval with }%
$\mu\left(  \left(  a,b\right)  \right)  <\frac{1}{2},a<0$, $b>0$ and
$a+b\geq0$ $($risp. $a+b\leq0)$ then
\[
P\left(  (a,b)\right)  \geq P\left(  (a^{\prime},b^{\prime})\right)
\]
for any left-shifted (risp. right-shifted) of $\left(  a,b\right)  $ with
$a^{\prime}+b^{\prime}\geq0$ (risp. $a^{\prime}+b^{\prime}\leq0$);

$2)$ if\textit{ }$\left(  a,b\right)  $\textit{ is an interval with }%
$\mu\left(  \left(  a,b\right)  \right)  \geq\frac{1}{2}$and $a+b\geq0$
$($risp. $a+b\leq0)$ then
\[
P\left((  a,b)\right)\geq P\left((  a^{\prime},b^{\prime
})\right)
\]
for any left-shifted (risp. right-shifted) of $\left(  a,b\right)  $ with
$a^{\prime}+b^{\prime}\geq0$ (risp. $a^{\prime}+b^{\prime}\leq0$).

Then $\mu$ is log-convex.
\end{proposition}

\begin{proof}
By Proposition \ref{prop mu-J}, continuity and symmetry of $J(t)$ we only have to prove
that, for all $t \in \left(  0,\frac{1}{2}\right)$ and all $d$ so that $t\pm d\in\left(  0,\frac{1}{2}\right)$,
it holds
\begin{equation}  \label{dis convex}
J(t)\geq\frac{1}{2}\left(  J(t-d)+J(t+d)\right)   .
\end{equation}
Fix $t\in\left(0,\frac{1}{2}\right)$ and $d$ such that $t\pm d\in\left(  0,\frac{1}{2}\right)$ and set
$a=F^{-1}(t)$, $b=F^{-1}(1-t)$, $a^{\prime}=F^{-1}(t+d)$ and $b^{\prime}=F^{-1}(1-t+d)$.
With these notations in hand, we observe that, by symmetry of $J$,
$$
P\left((  a,b)\right) = J(t)+J(1-t)=2J(t)
$$
and
$$
P\left( ( a^{\prime},b^{\prime})\right)
=
J(t+d)+J(1-t+d)=J(t+d)+J(t-d) .
$$
Then (\ref{dis convex}) precisely means that $P\left( ( a,b)\right) \geq P\left( ( a^{\prime},b^{\prime})\right)$
which is guaranteed by the shifting property assumption.
This ends the proof.
\end{proof}


\subsection{Isoperimetric problem for intervals and complement of intervals}

The geometric tool given in the previous section will allow us to answer the following warm-up isoperimetric problem: among all intervals (and then among all complement of intervals) of given measure, which one(s) has(have) minimal perimeter? The answer for intervals is stated in the next corollary: depending on the measure of the interval, the interval with minimal perimeter has to be found at infinity (half line), or centred around the origin.

We need a preliminary result.

\begin{lemma}\label{def:p0}
Let $\mu \in \mathcal{F}$. Then, there exists a unique $p_{0}\in\left(0,\frac{1}{2}\right)$ satisfying $J(1-p_0)=2J((1-p_0)/2)$
and such that $J(1-p) < 2J((1-p)/2)$ for $p \in [0,p_0)$ and $J(1-p) > 2J((1-p)/2)$ for $p\in(p_0,1/2]$, where $J(t)$ is defined in (\ref{J}).
\end{lemma}

\begin{remark}
In general $p_0$ is known only implicitly. However, in the case of the Cauchy measure $m_\alpha$ defined in \eqref{mis 1}, one easily sees that $p_{0}=\frac{1}{1+2^{1/(1+\alpha)}}$.
\end{remark}

\begin{proof}
Let us consider the auxiliary function $g(p)=J(1-p)-2J(\frac{1-p}{2})=J(1-p)-2J(\frac{1+p}{2})$ for $ p \in [0,\frac{1}{2}]$.
We observe that $g$ is continuous, increasing and $g(0)=-2J_{\mu
}(\frac{1}{2})<0$. Moreover, Lemma \ref{lemino} guarantees that $g(\frac{1}{2})=\frac{1}{2}\left(\frac{J(1/2)}{1/2}-\frac{J(1/4)}{1/4}\right)>0$. Hence the result.
\end{proof}

Introduce the following notation.
\begin{equation}
\alpha_{p}=-F^{-1}\left(\frac{1-p}{2}\right), \qquad  \sigma_{p}=-F^{-1}(p), \qquad  p \in [0,1]. \label{alfa/sigma}
\end{equation}

We are now in position to state the corollary.

\begin{corollary}[Extremal sets in the isoperimetric problem for intervals]\label{Prop intervalli}
Let $\mu \in \mathcal{F}$ and $p_0$ defined in Lemma \ref{def:p0}.
Let us fix $a,b$ with $-\infty\leq a<b\leq+\infty$ and set $p=\mu\left((a,b)\right)$.
Then, 
\begin{enumerate}[i)]
\item
if $p> p_{0}$,
\begin{equation} \label{isop intervalli 1}%
P\left( (  a,b)\right)  \geq P\left( ( -\alpha_{p},\alpha_{p})\right),
\end{equation}
with equality iff $\left(a,b\right)=\left(-\alpha_{p},\alpha_{p}\right)$;
\item
if $p < p_{0}$,
\begin{align} \label{isop int 2}
P\left((a,b)\right)
& \geq P\left((-\infty,-\sigma_{p})\right)  \\
& (= P\left((\sigma_{p},+\infty))\right),
\end{align}
with equality iff $\left(a,b\right)=\left(-\infty,-\sigma_{p}\right)$ or
 $\left(a,b\right)=\left(\sigma_{p},+\infty\right)$;
 \item if $p=p_0$,
 $$
 P\left((  a,b)\right)  \geq P\left((  -\alpha_{p},\alpha
_{p})\right) = P\left((-\infty,-\sigma_{p})\right)=P\left((\sigma_{p},+\infty)\right)
 $$
 with equality iff $\left(a,b\right)$ equals $\left(-\alpha_{p},\alpha_{p}\right)$,
 $\left(-\infty,-\sigma_{p}\right)$ or
 $\left(\sigma_{p},+\infty\right)$.
\end{enumerate}
\end{corollary}

\begin{proof}
If $p \geq 1/2$, the result of Point $(i)$ immediately follows
from Proposition \ref{shifting} Point $(2)$.

Hence, we need to deal with $0 \leq p<\frac{1}{2}$.
As in the proof of Proposition \ref{shifting}, we have
$$
P\left((a,b)\right)=J(F(a))+J(p+F\left(a\right)) = \psi_p(F(a)),
$$
where $\psi_p$ has been defined in Lemma \ref{lemino2}. Thanks to Lemma \ref{lemino2},
we need to compare the two minima of $\psi_p$, namely $\psi_p(1-p)=J(1-p)=P\left((-\infty,-\sigma_{p})\right)
=P\left((\sigma_{p},+\infty)\right)$ and
$\psi_p(\frac{1-p}{2})=2J(\frac{1-p}{2})=P\left((-\alpha_{p},\alpha_{p})\right)$.

By definition of $p_0$ (see Lemma \ref{def:p0}), we conclude that $a \mapsto \psi_{p}(F(a))$
has a unique global minimum at $a=F^{-1}\left(\frac{1-p}{2}\right)$ if $0<p< p_{0}$
and has two minima at $a=F^{-1}\left(0\right)$
and $a=F^{-1}\left(1-p\right)$ if $p_{0}< p<\frac{1}{2}$.
This ends the proof of the inequalities.
Equality cases follow at once from the strict monotonicity of $\psi_p$.
\end{proof}

\begin{remark}
Corollary \ref{Prop intervalli} implies that
\[
P\left((  a,b)\right)\geq\min\left\{  2J\left(  \frac{1-p}{2} \right),J(p)\right\}  ,
\qquad\quad
p\in\left[0,1\right]  .
\]
\end{remark}

%
%
%
%

\begin{remark}[Isoperimetric problem for complements of an interval]\label{prop 2 semirette}
\hspace{0cm}\\
Observing that intervals and complement of intervals have the same perimeter, \textit{i.e.}\ that
$P\left(  \left(  -\infty,a\right)  \cup\left(b,+\infty\right)  \right) =P\left(  \left(a,b\right)\right)$ one can
easily solve, using Corollary \ref{Prop intervalli}, the isoperimetric problem among sets of prescribed measure, that are
complements of an interval. More precisely, one obtains the following (details are left to the reader):
Fix $a,b$ with $-\infty\leq a<b\leq+\infty$ and set $p=\mu\left(\left(-\infty,a\right)  \cup\left(  b,+\infty\right)\right)$.
Then (using notations of Corollary \ref{Prop intervalli}),
\begin{enumerate}[i)]
\item
if $p>1-p_{0}$,
\begin{align} \label{isop due semirette}
P\left(( -\infty,a)  \cup( b,+\infty)\right)
& \geq
P\left((-\infty,\sigma_{1-p})\right) \\
& (=
P\left((-\sigma_{1-p},+\infty)\right)) \nonumber
\end{align}
with equality iff $\left( -\infty,a)  \cup( b,+\infty)\right) = \left(-\infty,\sigma_{1-p}\right)$
or $\left( -\infty,a\right)  \cup\left( b,+\infty\right)=\left(-\sigma_{1-p},\infty\right)$;
\item if $p< 1-p_{0}$,
\begin{equation}  \label{isop compl 2}
 P\left((-\infty,a)  \cup (b,+\infty) \right)
 \geq
 P\left((-\infty,-\alpha_{1-p})  \cup (\alpha_{1-p},+\infty)\right)
\end{equation}
with equality iff $\left(-\infty,a\right)  \cup\left(b,+\infty\right)=\left(-\infty,-\alpha_{1-p}\right)  \cup \left(\alpha_{1-p},+\infty\right)$;
\item if $p=1-p_0$,
\begin{align*}
P\left((-\infty,a)  \cup(b,+\infty)\right)
& \geq
 P\left((-\infty,-\alpha_{1-p})  \cup (\alpha_{1-p},+\infty)\right) \\
& =
 P\left((-\infty,\sigma_{1-p})\right)
=
P\left((-\sigma_{1-p},+\infty)\right)
\end{align*}
with equality iff $\left(-\infty,a\right)  \cup\left(b,+\infty\right)$ equals $\left(-\infty,\sigma_{1-p}\right)$,
$\left(-\infty,\sigma_{1-p}\right)$ or the set $\left(-\infty,-\alpha_{1-p}\right)  \cup \left(\alpha_{1-p},+\infty\right)$.
\end{enumerate}
\end{remark}


\subsection{Isoperimetric inequality for strictly log-convex probability measures}

\hspace{0,1cm}\\ From the results of the previous sections, we can now solve the isoperimetric problem for
strictly log-convex probability measures.

In what follows we need to recall the definition of $\alpha_p$ and $\sigma_p$ given in (\ref{alfa/sigma}). For simplicity, we set also
$\beta_p=\alpha_{1-p}$, for $p \in [0,1]$.

\begin{theorem}[Isoperimety for strictly log-convex probability measures] \label{th generale}
Let $\mu \in \mathcal{F}$
and $E$ be a Borel set of $\mathbb{R}$ with measure $\mu(E)=p$.
Then
\begin{enumerate}[i)]

\item if $p<1/2$,
\begin{equation} \label{isop 2}%
P(E)\geq P\left((  -\infty,-\beta_{p})  \cup(
\beta_{p},+\infty)\right),
\end{equation}
with equality iff
$E=\left(-\infty,-\beta_{p})  \cup(  \beta_{p},+\infty)\right)$;

\item if $p>1/2$,
\begin{equation}
P(E)\geq P\left((  -\alpha_{p},\alpha_{p})\right)  \label{isop 1}
\end{equation}
with equality iff  $E=\left(-\alpha_{p},\alpha_{p}\right)$;

\item if $p=\frac{1}{2}$
\begin{equation}
P(E)\geq P\left((  -\infty,-\beta_{p})  \cup(
\beta_{p},+\infty)\right)  =P\left((  -\alpha_{p},\alpha_{p})\right)
\label{isop 3}%
\end{equation}
with equality iff  $E$ equals $\left(-\infty,-\beta_{p}\right)\cup\left(\beta_{p},+\infty\right)$ or
$\left(-\alpha_{p},\alpha_{p}\right).$
\end{enumerate}
\end{theorem}

\begin{remark}
The measure $1/2$ can be seen as an \emph{isoperimetric threshold}, in the sense that extremal sets
move from complement of symmetric intervals (when $p<1/2$) to symmetric intervals (when $p>1/2$).
\end{remark}

\begin{proof}
Let $E$ be a Borel set with measure $p=\mu(E)$.
Without loss of generality we can assume that $E$ has finite perimeter.

We begin with a simple remark. Assume that
$E=\left(  a_{1},b_{1}\right)\cup \left(a_{2},b_{2}\right)$ with
$a_{1}<b_{1}<a_{2}<b_{2}$ and $a_{1}>0$. Then, consider the right shifted interval
$(a'_1,b'_1)$ of $(a_1,b_1)$, with $b'_1=a_2$ so that, thanks to Proposition \ref{shifting},
$$
P(E) \geq P \left( \left(  a_{1}^{\prime},b_{1}^{\prime}\right)
\cup \left(  a_{2},b_{2}\right) \right) \geq P\left(\left(  a_{1}^{\prime}
,b_{2}\right)  \right) .
$$
In conclusion, we get that, given two disconnected intervals contained in $(0,\infty)$, the perimeter decreases
by moving (and gluing) the left most interval toward the right most one. Clearly, the same property holds for $a_{1}<b_{1}<a_{2}<b_{2}<0$ by symmetry.

Now, let $E$ be any set of finite $\mu-$perimeter. From standard measure Theory (see \textit{e.g.}\ \cite[Proposition 3.52]{libro-fusco})
there exists a countable
set $H$ such that, up to a set of measure zero,
$E=\underset{h\in H}{\cup}\left(  a_{h},b_{h}\right)$, where $-\infty\leq a_{h}<b_{h}\leq+\infty$ and
$\mathrm{dist}\left(E \setminus \left( a_{h},b_{h}\right),\left(a_{h},b_{h}\right)\right)>0$ for all $h\in H$.
Without loss of generality we can assume that the set of measure zero is the empty set.
Then, iterating the remark above (for two intervals), we obtain that, either
\[
P(E)\geq P\left((  -\infty,-a)  \cup(b,+\infty)\right)
\]
 if $0 \notin E$, or, if $0 \in\left(a_{h_o},b_{h_o}\right)$ for some $h_o \in H$,
\[
P(E)\geq P\left((  -\infty,-\overline{a})\cup(  a_{h_o},b_{h_o})
\cup (  \overline{b},+\infty)\right),
\]
where $a,\bar a, b,\bar b \in (0,\infty]$ and $\left( -\infty,-a\right)
\cup\left(b,+\infty\right)$ and $\left(-\infty,-\overline{a}\right)
\cup\left(  a_{h_o},b_{h_o}\right) \cup\left(\overline{b},+\infty\right)$ are
sets with measure $p=\mu(E)$.
Here and below we use the convention that
$(\infty,\infty)=(-\infty,-\infty)=\emptyset$.
In the second case (\textit{i.e.}\ when $0 \in E$), we continue the reduction by considering the complementary set
$$
\left[\left(  -\infty,-\overline{a}\right)
\cup\left(  a_{h_o},b_{h_o}\right)  \cup\left(  \overline{b},+\infty\right)\right]^c
= [-\overline{a},a_{h_o}] \cup [b_{h_o}, \bar b],
$$
which has the same measure and same perimeter as $\widehat{E}:=(-\overline{a},a_{h_o}) \cup (b_{h_o}, \bar b)$ that we will deal with.
By construction and since $\mu$ is symmetric, necessarily
$\mu((-\overline{a},a_{h_o})) < 1/2$ and
$\mu((b_{h_o}, \bar b)) < 1/2$. Hence, thanks to the shifting property of Proposition \ref{shifting} (Point (1)),
$$
P\left( (-\overline{a},a_{h_o}) \right) \geq P\left( (-\infty,-\alpha) \right)
$$
and
$$
P\left( (b_{h_o}, \bar b) \right) \geq P\left( (\beta,\infty) \right),
$$
where $\alpha, \beta \in (0,\infty]$ are such that $\mu\left( (-\infty,-\alpha) \right)= \mu \left( (-\overline{a},a_{h_o}) \right)$ and  $\mu\left( (\beta,\infty) \right)=\mu\left( (b_{h_o}, \bar b) \right)$.
Going back to the complementary set, we end up with the following bound
\begin{align*}
P(E)
& \geq
P(\widehat{E})
 \geq
P \left( (-\infty,-\alpha) \cup (\beta,\infty)\right)
=
P \left( (-\alpha,\beta)\right)
\end{align*}
with $\mu\left( (-\alpha,\beta)\right)=p$.

As a summary, after few reductions, we obtained the following two cases: either
$$
(I)\quad P(E)\geq P\left((  -\infty,-a)  \cup(
b,+\infty)\right) \text{  if } 0 \notin E
$$
 or
$$
(II)\quad P(E)\geq P \left( (-\alpha,\beta)\right) \text {  if } 0 \in E,
$$
 where $\mu\left((  -\infty,-a)  \cup(b,+\infty)\right)  = \mu \left( (-\alpha,\beta)\right)=\mu(E)=p$
and $a, \alpha, b,\beta \in (0,\infty]$.


\medskip

Now assume that $p \in(0,1/2]$. We distinguish between cases $(I)$ and $(II)$.

\textit{Case (I).} Applying Remark \ref{prop 2 semirette} Point $ii)$ (observe that, since $p \leq 1/2$, ne\-ces\-sa\-ri\-ly $p < 1-p_0$, where $p_0$ is defined in Lemma \ref{def:p0}),
the perimeter decreases if we consider the symmetric set
$\left(-\infty,-\alpha_{1-p}\right) \cup\left(\alpha_{1-p},+\infty\right)$,
unless $E=\left(-\infty,-\alpha_{1-p}\right)  \cup\left(  \alpha_{1-p},+\infty\right)$,
where we recall that $\alpha_{1-p}=-F^{-1}(p/2)=\beta_p$.

As a conclusion, in case $(I)$,
$P(E) \geq P\left(\left(-\infty,-\beta_p\right)  \cup\left(  \beta_p,+\infty\right)\right)$.

\textit{Case (II).}
Corollary \ref{Prop intervalli} (Point $ii)$) guarantees that
\begin{align*}
P(E)
& \geq
P \left( (-\alpha,\beta)\right)
\geq
P \left( (-\infty,-\sigma_{p})\right) = J(p)
\end{align*}
(where $\sigma_p=-F^{-1}(p)$). Lemma \ref{lemino} implies that,
for $p \in [0,\frac{1}{2}]$, $J(p) \geq 2J(\frac{p}{2})=P \left((-\infty,-\beta_p)  \cup(  \beta_p,+\infty)\right)$.
The inequality of Point $i)$ follows. Keeping track of the equality cases in the various steps above leads
to the desired result of Point $i)$ and $iii)$

Finally, point $ii)$ is an easy consequence of Point $i)$ considering the complementary set.
\end{proof}

\begin{remark} \label{explicit}
Notice that, if $\mu$ is not strictly log-convex then equality cases in the above theorem no longer hold (see next subsection 3.4).

Also, observe that Theorem \ref{th generale}\ gives the following explicit expression of the isoperimetric
profile, recovering \cite{bob-hudre} for strictly log-convex probability measures,
$$
I(p)
=
2J\left(  \frac{1}{2}%
\min\left(  p,1-p\right)  \right),
\qquad p\in\left[  0,1\right]  .
$$
\end{remark}


\subsection{Example of the two-sided exponential measure}

In this subsection, we brie\-fly deal with an example of non strictly log-convex probability measure,
the two-side exponential measure defined in \eqref{misura 0}.
It is a symmetric probability measure, log-convex \emph{and} log-concave,
with $J(t)=\min\left(  t,1-t\right)$.

The perimeter $P_{\mu_{1}}\left(( a,b)\right)$ of an interval $\left(a,b\right)$
of fixed measure $p$ can be explicitly computed. If $p\geq\frac{1}{2}$ then
$P_{\mu_{1}}\left((a,b)\right)=1-p$.
In particular, we stress that all intervals of measure bigger than $1/2$
have the same perimeter. If $p<\frac{1}{2}$ we have
\[
P_{\mu_{1}}\left((a,b)\right) = P_{p}(a)=\left\{
\begin{array}
[c]{lll}%
2F\left(  a\right)  +p & \text{if} & -\infty\leq a\leq F^{-1}\left(
\frac{1}{2}-p\right),  \\
&  & \\
1-p & \text{if} & F^{-1}\left(  \frac{1}{2}-p\right)  \leq a\leq
0,\\
&  & \\
2-2F\left(  a\right)  -p & \text{if} & 0  \leq a\leq
F^{-1}\left(  1-p\right).
\end{array}
\right.
\]
Therefore among the intervals $\left(  a,b\right)  $ of measure $p<\frac{1}%
{2}$ the half-lines have minimal perimeter. Moreover the shifting property for
interval is the following:

\begin{proposition}
If\textit{ }$\left(  a,b\right)  $\textit{ is an interval of measure such that
}$\mu_{1}\left(  \left(  a,b\right)  \right)  <\frac{1}{2}$,  then
\[
P_{\mu_{1}}\left((  a,b)\right)>P_{\mu_{1}}\left((  a^{\prime},b^{\prime})\right)
\]
for any left-shifted (right-shifted) interval of $\left(  a,b\right)  $ with
$b^{\prime}\leq0$ ($a^{\prime}\geq0$). Otherwise%
\[
P_{\mu_{1}}\left((  a,b)\right)=P_{\mu_{1}}\left((  a^{\prime},b^{\prime})\right)
\]
for any right-shifted or left-shifted interval of $\left(  a,b\right)  .$
\end{proposition}

Arguing as in Theorem \ref{th generale} we obtain that for a fixed measure $p<\frac{1}{2}$
any complement of an interval and the half-lines are sets with minimal
perimeter. For $p\geq\frac{1}{2}$ any interval and half-lines have minimal
perimeter.


\section{Quantitative isoperimetric inequality}

In this section, following \cite{fusco}, we introduce and study
a notion of asymmetry, which quantify the "distance" between any measurable set
$E$ and the family of extremal sets in the isoperimetric problem.
Then we state a preliminary result on the sets that have minimal
perimeter and given measure and asymmetry
.

We define the asymmetry $\lambda(E)$ of a set $E$ of measure $p=\mu(E)$ as%
\begin{equation} \label{asim}
\lambda(E)=
\begin{cases}
\mu\left(E \bigtriangleup ( -\infty,-\beta_{p})  \cup(  \beta
_{p},+\infty)\right)   & \text{if }  p<\frac{1}{2}\\
\mu\left (E\bigtriangleup (  -\alpha_{p},\alpha_{p})\right)   & \text{if }
p>\frac{1}{2}\\
\min\left\{  \mu \left(E\bigtriangleup (  -\infty,-\beta_{p} )  \cup (
\beta_{p},+\infty)\right)  ,
\mu\left(E\bigtriangleup (  -\alpha_{p},\alpha_{p})  \right)\right\} \hspace{-0,2cm} & \text{if }  p=\frac{1}{2},
\end{cases}
\end{equation}
where $\beta_{p}=\alpha_{1-p}=-F^{-1}(p/2)$, $p \in [0,1]$ are defined in the previous section
and $\bigtriangleup$ stands for the symmetric difference between sets.

\begin{remark}
To help the reader in many computations throughout all this section, we observe that the set $E=(-\beta_a,-\beta_b)$, with $0 \leq a \leq b$, has perimeter $P(E)=J(b/2)+J(a/2)$ and measure
$\mu(E)=(b-a)/2$.
\end{remark}

The next lemma summarizes some basic properties on the asymmetry $\lambda(E)$.

\begin{lemma} \label{sera}
Let $\mu \in \mathcal{F}$ and $E, F$ be two sets with finite $\mu$-perimeter.
Then,
\begin{itemize}
\item[i)] $E\bigtriangleup F=E^{c}\bigtriangleup F^{c},$
\item[ii)] $\lambda(E)=\lambda(E^{c}),$
\item[iii)] $0\leq\lambda(E)\leq2\min( \mu(E),1-\mu(E))$,
\end{itemize}
where $E^{c}$ denotes the complement of $E$.
\end{lemma}

\begin{proof}
The assertions
are easy and left to the reader.

\end{proof}


\subsection{A preliminary reduction and application} \label{sec:preliminary}

The next result is a first reduction to find the sets with minimal
perimeter and given measure and asymmetry. We first consider the case $0<\mu(E)\leq \frac{1}{2}$. The other case $\frac{1}{2}\leq\mu(E)<1$ can be obtained using complementary sets and Lemma \ref{sera}.

Given a set $E$, we set $\bar{E}=\{-x,x\in E\}$ for its symmetric with respect to the origin.

We will show, in Proposition \ref{prop quant 1} below, that the minimal sets among all sets of given measure $p$ and given asymmetry $\lambda$ have to be found among the following sets and their symmetric:
\begin{align}
& E_{1}
 =
\left(-\beta_{\frac{\lambda}{2}},-\beta_{p+\frac{\lambda}{2}}\right)
\cup \left(\beta_{p+\frac{\lambda}{2}},\beta_{\frac{\lambda}{2}}\right),
&&\mbox{if } 0\leq\lambda\leq2p , \label{E1} \\
& E_{2}
=
\left(-\infty,-\beta_{p+ \lambda}\right)\cup\left(\beta_{p-\lambda},+\infty\right) \mbox{ and } \bar{E_2},
&&\mbox{if } 0\leq\lambda\leq p, \label{E2} \\
& E_{3}
 =
\left(-\beta_{\lambda-p},-\beta_{\lambda+p}\right)  \mbox{ and }  \bar{E_3},
&&\mbox{if } p\leq\lambda\leq 2p, \label{E3}
\end{align}
and,  if $0\leq\lambda\leq 2p$,
\begin{equation}
 E_{4}
 =
\left(-\infty,-\beta_{p-\frac{\lambda}{2}}\right)
\cup
\left(-\beta_{1-\frac{\lambda}{2}},\beta_{1-\frac{\lambda}{2}}\right)
\cup
\left(\beta_{p-\frac{\lambda}{2}},+\infty \right)
.  \label{E4}
\end{equation}
Observe that $E_2$ and $\bar{E_2}$ are not defined when $\lambda > p$ and that $E_3$ and $\bar{E_3}$ are not defined when $\lambda < p$.


\begin{figure}[ht]
\psfrag{e1}{$E_4$}
\psfrag{e2}{$E_3$}
\psfrag{e3}{$E_2$}
\psfrag{e4}{$E_1$}
\psfrag{0}{$0$}
\psfrag{-i}{$-\infty$}
\psfrag{i}{$+\infty$}
\psfrag{-bp}{$-\beta_p$}
\psfrag{bp}{$\beta_p$}
\psfrag{1}{$-\beta_{\frac{\lambda}{2}}$}
\psfrag{2}{$-\beta_{p+\frac{\lambda}{2}}$}
\psfrag{3}{$\beta_{p+\frac{\lambda}{2}}$}
\psfrag{4}{$\beta_{\frac{\lambda}{2}}$}
\psfrag{5}{$-\beta_{p+\lambda}$}
\psfrag{6}{$\beta_{p-\lambda}$}
\psfrag{7}{$-\beta_{\lambda-p}$}
\psfrag{8}{$-\beta_{p+\lambda}$}
\psfrag{9}{$-\beta_{p-\frac{\lambda}{2}}$}
\psfrag{10}{$-\beta_{1-\frac{\lambda}{2}}$}
\psfrag{11}{$\beta_{1-\frac{\lambda}{2}}$}
\psfrag{12}{$\beta_{p-\frac{\lambda}{2}}$}
\begin{center}
\includegraphics[width=.80\columnwidth]{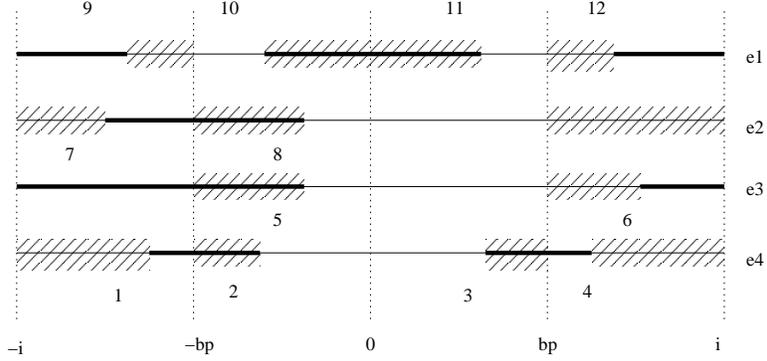}
\end{center}
\caption{The sets $E_1,\dots,E_4$ (bold lines). Hachured parts (of measure $\lambda$) correspond to the symmetric difference between the set $E_i$ and the optimal set $(-\infty,-\beta_p) \cup (\beta_p,\infty)$.
Observe that there is no universal order between the extremal points of the intervals defining
$E_1,\dots,E_4$, except that $\beta_{\lambda-p} \geq \beta_{\lambda/2}$, and $\beta_{p-\lambda} \geq \beta_{p-\frac{\lambda}{2}}$. Also, depending on the value of $p$ and $\lambda$, it could be that $0 \in E_3$.
}
\end{figure}

\begin{proposition} \label{prop quant 1}
Let $\mu \in \mathcal{F}$
and $E$ be a Borel set with measure $\mu(E)=p \in (0,\frac{1}{2}]$ and asymmetry $\lambda(E)=\lambda \in [0,2p]$. Then,
\begin{equation} \label{e tilde}
P(E)\geq
\begin{cases}
\underset{i=1,2,4}{\min}{P(E_{i})} & \text{if } 0\leq\lambda\leq p \\
\underset{i=1,3,4}{\min}{P(E_{i})} & \text{if } p \leq\lambda\leq 2p .
\end{cases}
\end{equation}
Moreover equality holds if and only if $E \in \{E_1,E_2,E_3,E_4,\bar{E_2},\bar{E_3}\}$.

If $0 \notin E$, then
\begin{equation*}
P(E)\geq
\begin{cases}
\underset{i=1,2}{\min}{P(E_{i})} & \text{if } 0\leq\lambda\leq p \\
\underset{i=1,3}{\min}{P(E_{i})} & \text{if } p \leq\lambda\leq \min(1-p,2p) \\
P(E_1) & \text{if } 1-p < \lambda\leq 2p
\end{cases}
\end{equation*}
with equality if and only if $E \in \{E_1,E_2,E_3,\bar{E_2},\bar{E_3}\}$.
\end{proposition}

\begin{remark}
The second part of the above proposition (together with Proposition \ref{prop quant}  below)
will be used in Section \ref{sec:functional} where we will only consider
sets that do not contain the origin.
\end{remark}

\begin{proof}
Let $E$ be a set of measure $p \in (0,\frac{1}{2})$ and asymmetry $\lambda$.
As in  the proof of Theorem \ref{th generale} there exists a countable set $H$ such that $E=\underset{h\in H}{\cup}\left(  a_{h},b_{h}\right)$ up to a set of measure zero that we assume, without loss of generality, to be the empty set.

\textit{Step 1.} Arguing as in the proof of Theorem \ref{th generale} and using the shifting property (of Proposition \ref{shifting}), preserving not only the measure of the set but also its asymmetry (see Figure \ref{fig:reduction2} for an illustration), we  obtain (details are left to the reader)
 that
\begin{equation}\label{first step}
P(E)\geq P \left( \widetilde{E} \right),
\end{equation}
where
\[
\widetilde{E}=(-\infty,a_3) \cup (a_2,a_1) \cup(  a_0,b_0)\cup (b_1,b_2) \cup (b_3,+\infty)
\]
and
\begin{equation} \label{range1}
-\infty\leq a_3\leq a_2\leq  -\beta_{p} \leq a_1 \leq a_0\leq0\leq b_0\leq b_1\leq\beta_{p}\leq b_2\leq b_3\leq +\infty .
\end{equation}

Depending on the initial set $E$ every inequalities in \eqref{range1} might be either an equality or a strictly inequality.
In case of equality, we convey that $(\alpha,\alpha)=\emptyset$ for any $\alpha \in [-\infty,+\infty]$.

\begin{figure}[ht]
\psfrag{-b}{\tiny $-\beta_p$}
\psfrag{-b'}{\tiny $a_2$}
\psfrag{bp}{\tiny $\beta_p$}
\psfrag{-a}{\tiny $a_3$}
\psfrag{-a'}{\tiny $a_1$}
\psfrag{-a''}{\tiny $a_0$}
\psfrag{0'}{\tiny $a_0=0=b_0$}
\psfrag{b'}{\tiny $b_1$}
\psfrag{b''}{\tiny $b_3$}
\psfrag{b}{\tiny $b_2$}
\psfrag{-c}{\tiny $-a$}
\psfrag{-c'}{\tiny $-a'$}
\psfrag{d}{\tiny $b$}
\psfrag{d'}{\tiny $b'$}
\psfrag{0}{\tiny $0$}
\begin{center}
\includegraphics[width=.90\columnwidth]{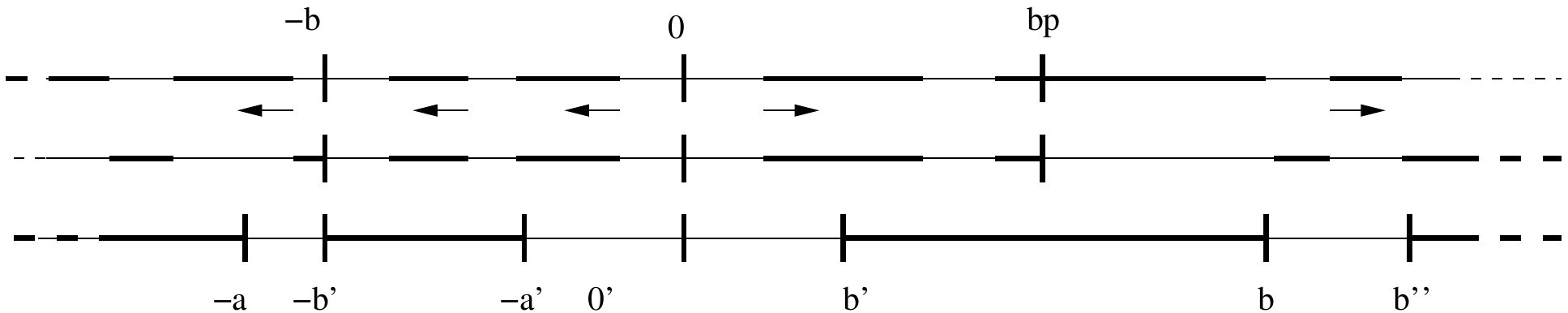}
\end{center}
\caption{The reduction from $E$, when $0 \notin E$, to the set
$\widetilde{E}$.
The top line represents the set $E=\underset{h\in H}{\cup}\left(  a_{h},b_{h}\right)$,
the second line is the symmetric difference $E \bigtriangleup (-\infty,-\beta_p)\cup(\beta_p,\infty)$.
The arrows show how to use the shifting property, preserving the asymmetry.
The bottom line is the set $\widetilde{E}=(-\infty,a_3) \cup (a_2,a_1) \cup(  a_0,b_0)\cup (b_1,b_2) \cup (b_3,+\infty)$
in the particular case where $a_2=-\beta_p$ and $a_0=b_0=0$.
}
\label{fig:reduction2}
\end{figure}

\textit{Step 2.} Examining all\footnote{there are $2^{12}$ of them (but a lot of symmetries!).} the possibilities in \eqref{range1} (equality versus strict inequality)
and using again the shifting property, it is possible to further reduce the family of sets with minimal perimeter.
Indeed, after reduction (which is an easy (but tedious) exercise left to the reader) one concludes that the minimal perimeter
has to be found among only the following 7
sets (see below for an example of such a reduction): $E_{1}$,  $E_{2}$ and $\bar{E_2}$, $E_{3}$ and $\bar{E_3}$, $E_{4}$, defined in (\ref{E1})-(\ref{E4}), and
$$
E_5 =
\left(-\infty,-\beta_{2p-\lambda}\right)
\cup \left(-\beta_{1-\frac{\lambda}{2}},\beta_{1-\frac{\lambda}{2}}\right)
 \text{ and } \bar{E_5}
\qquad \text {if } p\leq\lambda\leq 2p ,
$$
if either $0\leq\lambda\leq p$ and any given $t \in [0, p]$
or $p \leq \lambda \leq 2p$   and any given $t \in [ \lambda-p ,2p-\lambda]$,
$$
E_6=
\left(-\beta_{p-t},-\beta_{p+\frac{\lambda}{2}}\right)
\cup\left(\beta_{p+\frac{\lambda}{2}},\beta_{-p+\lambda+t} \right) \text{ and } \bar{E_6} ,
$$
and, if either $0\leq\lambda\leq p$ and any given $t \in [0,\lambda]$
or $p \leq \lambda \leq 2p$   and any given $t \in [\lambda -p , p]$,
$$
E_{7}=
\left(-\beta_{\lambda-t},-\beta_{p+\lambda}\right)
\cup\left(\beta_{p-t},+\infty \right) \text{ and } \bar{E_7} .
$$
As an example of the above reduction, let us consider the set $\widetilde{E}$ with
$-\infty<a_3 <a_2<a_1<  a_0=b_0 < b_1<b_2<b_3<+\infty$ in \eqref{range1}.
Consider the complementary set
$\mathbb{R} \setminus \widetilde{E} =  (a_3,a_2) \cup (a_1,b_1) \cup (b_2,b_3)$. Using the shifting property,
the perimeter decreases if one moves the interval $(a_3 ,a_2)$ towards $-\infty$ and the interval $(b_2,b_3)$ towards $+\infty$. Furthermore, the shifting property also guarantees that the perimeter decreases if one symmetrizes the interval $(a_1, b_1)$. All such reductions did not affect neither the measure nor the asymmetry. Finally,
considering again the complementary set, we end up with the set $E_{1}$ defined (\ref{E1}).

\textit{Step 3.} At this step, the shifting property  becomes useless.
To end the proof, one needs to show that $E_5, E_6$ and $E_7$ have bigger perimeter than
$E_1,E_2,E_3,E_4$ which will be achieved by using simple analytical computations.

First, we observe that
$$
P(E_5)=P\left((-\infty,-\beta_{2p-\lambda})\right) + P((-\beta_{1-\frac{\lambda}{2}},\beta_{1-\frac{\lambda}{2}}))
> P(E_4)
$$
since, by the isoperimetric inequality (Theorem \ref{th generale}), we are guaranteed that\newline
$P\left((-\infty,-\beta_{2p-\lambda})\right)>P((-\infty,-\beta_{p-\frac{\lambda}{2}})\cup(\beta_{p-\frac{\lambda}{2}},\infty))$.

We notice that $|\frac{\lambda}{4} - \frac{-p+\lambda+t}{2}| = |\frac{p-t}{2}-\frac{\lambda}{4}|$
so that, by convexity of $J$ (comparing the slopes),
\begin{align*}
 J\left(\frac{p-t}{2}\right) + & J\left(\frac{-p+\lambda+t}{2}\right) - 2J\left(\frac{\lambda}{4}\right) \\
& =
\begin{cases}
\left(J(\frac{p-t}{2}) -J(\frac{\lambda}{4})\right) - \left(J(\frac{\lambda}{4})- J(\frac{-p+\lambda+t}{2})\right) & \mbox{if } t \leq p- \frac{\lambda}{2} \\
\left(J(\frac{-p+\lambda+t}{2})- J(\frac{\lambda}{4})\right) - \left(J(\frac{\lambda}{4})- J(\frac{p-t}{2})\right)  & \mbox{if } t \geq p- \frac{\lambda}{2} \\
\end{cases} \\
&
\geq 0
\end{align*}
which, in turn, immediately implies that
$P(E_6) \geq P(E_1)$.

Finally, we observe that the map
$t \mapsto P(E_7)=J\left(\frac{\lambda-t}{2}\right)+J\left(\frac{p+\lambda}{2}\right) + J\left(\frac{p-t}{2}\right)$
is decreasing, so that (take $t=0$ and $t=\lambda-p$ respectively, which gives the same result)
\begin{align*}
P(E_7)
& \geq
J\left(\frac{\lambda}{2}\right)+J\left(\frac{p+\lambda}{2}\right) + J\left(\frac{p}{2}\right) \\
& \geq
\begin{cases}
J\left(\frac{p+\lambda}{2}\right) + J\left(\frac{p-\lambda}{2}\right) = P(E_2) & \mbox{if } 0 \leq \lambda \leq p \\
J\left(\frac{p+\lambda}{2}\right) + J\left(\frac{\lambda-p}{2}\right) = P(E_3) & \mbox{if } p \leq \lambda \leq 2p .
\end{cases}
\end{align*}

This completes the proof  of the first part of the proposition (equality cases follows easily by keeping track of equality cases in the various steps
in the reduction above).

The second part follows the same lines. One only needs to observe that $E_4$ need not be considered, since $0 \in E_4$, and that $0 \notin E_3$ implies $\lambda \leq 1-p$ (hence the range $p \leq \lambda \leq 1-p$). Also, observe that using the shifting lemma never affects the fact that $0 \notin E$ during the various step of the reduction above.
This achieves the proof.
\end{proof}

\bigskip
 At this point it is not possible to conclude which one of the sets $E_{i}$, $i=1,..,4$ has minimal perimeter on the range
 $p \in [0,1/2]$, $\lambda \in [0,2p]$ for the whole class of probability measures $\mathcal{F}$.
 Indeed, depending on the choice of $\mu \in \mathcal{F}$, one can exhibit very different behaviors.

To illustrate this phenomenon, let us deal with two specific generalized Cauchy distributions \eqref{mis 1}, with parameter $\alpha=1$ and $\alpha=1/2$:
$$
dm_1(x)=\frac{1}{2(1+|x|)^2}dx \qquad \mbox{and} \quad dm_{1/2}(x)= \frac{1}{4(1+|x|)^{3/2}}dx .
$$
Recall that $J_\alpha(t)=\alpha2^{\frac{1}{\alpha}} \min(t,1-t)^{1+\frac{1}{\alpha}}$ so that
$$
J_1(t)=2\min(t,1-t)^2 \qquad \mbox{and} \quad J_{1/2}(t) = 2\min(t,1-t)^3 .
$$
Since functions $J$ are explicit, one can compute the various perimeter $P(E_1)$, $P(E_2)$, $P(E_3)$, $P(E_4)$ and compare them.
It is simple (but very tedious, details are left to the reader) to obtain Figure \ref{fig:cauchy} below that depicts, for $p \in [0,1/2]$, the different regions with minimal perimeter (note that the region $p \in[1/2,1]$ can be obtained by symmetry about $p=1/2$, using Lemma \ref{sera}).
\begin{figure}[ht]
\psfrag{l1}{\tiny $\lambda=1-p$}
\psfrag{l2}{\tiny $\lambda=2p$}
\psfrag{l3}{\tiny $\lambda=p$}
\psfrag{e1}{\tiny $E_1$}
\psfrag{e2}{\tiny $E_2$}
\psfrag{e3}{\tiny $E_3$}
\psfrag{e4}{\tiny $E_4$}
\psfrag{l}{\tiny $\lambda$}
\psfrag{12}{\tiny $\frac{1}{2}$}
\psfrag{1}{\tiny $1$}
\psfrag{0}{\tiny $0$}
\psfrag{a1}{$\alpha=1$}
\psfrag{a2}{$\alpha=\frac{1}{2}$}
\begin{center}
\includegraphics[width=.80\columnwidth]{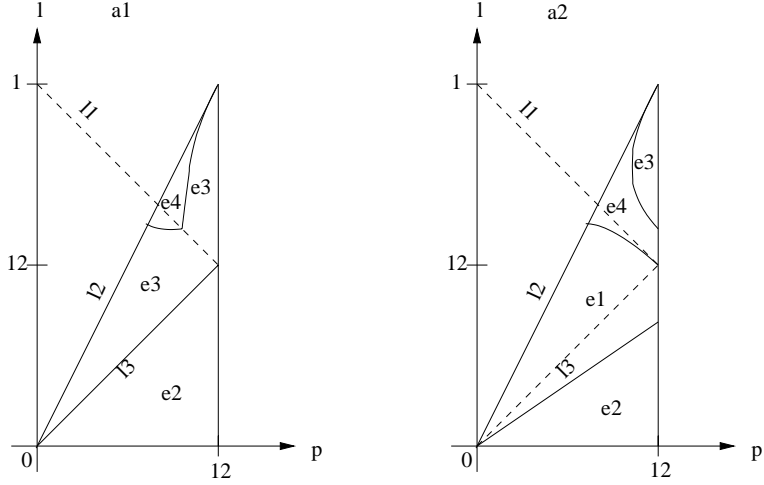}
\end{center}
\caption{The region $p \in \left[0,\frac{1}{2}\right]$, $0 \leq \lambda \leq 2p$ with the different areas where the set $E_i$
has minimal perimeter among $E_1,E_2,E_3$ and $E_4$.
The picture on the left corresponds to the Cauchy distribution with $\alpha=1$. Here, the $E_4$-domain is delimited by two curves of equation $\lambda=-1-p+\sqrt{3+2p+p^2}$ (bottom, that intersect the line $\lambda=2p$ at $p_1=\frac{\sqrt{5}-1}{4}$) and
$-1+2p+\lambda(1-p)-\frac{\lambda^2}{2}=0$ (top). The two curves intersect on the line of equation $\lambda=1-p$ at $p_2=\sqrt{2}-1$.
The picture on the right corresponds to $\alpha=\frac{1}{2}$. The $E_2$-domain and $E_1$-domain are delimited by a straight line of equation $\lambda=\frac{4\sqrt{3}p}{3\sqrt{3}+\sqrt{19}}$. The two curves delimiting the $E_3,E_4$ region are degree 3  polynomials in $p,\lambda$.}
\label{fig:cauchy}
\end{figure}

If $m_1$ and $m_{1/2}$ have very different behaviors, under additional assumptions,
one can however prove that a large sub-class of $\mathcal{F}$
behaves like $m_1$ (\textit{i.e.} have the same type of picture than the left one in Figure \ref{fig:cauchy}).
This is stated in the next proposition.

\begin{proposition} \label{prop quant}
Let $\mu \in \mathcal{F}$ and assume in addition that $J \in \mathcal{C}^1\left(0,\frac{1}{2}\right)$,
$J'$ is concave on $\left(0,\frac{1}{2}\right)$ and $J'(0^{+})=0$.
Let us fix $p \in \left[0,\frac{1}{2}\right]$ and $\lambda \in [0,1]$, and let us define $E_2,E_3$ and $E_4$ as in \eqref{E2},\eqref{E3} and \eqref{E4} respectively.

Then, there exist $p_{1}\in \left(0,\frac{1}{3}\right)$, $p_{2} \in \left(\frac{1}{3},{1}{2}\right)$, a function $\lambda_0:[p_1,p_2] \to [0,1]$ satisfying $\lambda_0(p_1)=2p_1$ and $\lambda_0(p_2)=1-p_2$, and a $\mathcal{C}^1$-increasing function
$p_0 : [1-p_2,1] \to \left[0,\frac{1}{2}\right]$ satisfying $p_0(1-p_2)=p_2$ and $p_0'(1)=\frac{1}{2}$, 
such that for any Borel set $E$ with measure $p$ and asymmetry $\lambda$ it holds
$$
P(E) \geq
\begin{cases}
P(E_2) & \mbox{if } 0 \leq \lambda \leq p \\
P(E_4) & \mbox{if } p \in [p_1,p_2] \mbox{ and } \lambda \in [\lambda_0(p),\min(2p,1-p)] \\
P(E_4) & \mbox{if } \lambda \in [1-p_2,1] \mbox{ and } p \in \left[\frac{\lambda}{2},p_0(\lambda)\right] \\
P(E_3) & \mbox{otherwise}.
\end{cases}
$$
Moreover, if $0 \notin E$, then
\begin{equation}
P(E) \geq \label{range}
\begin{cases}
P(E_2) & \mbox{if } 0 \leq \lambda \leq p \\
P(E_3) & \mbox{if } p \leq \lambda \leq \min(1-p,2p)\\
P(E_1) & \mbox{if } 1-p < \lambda \leq 2p .
\end{cases}
\end{equation}
\end{proposition}

\begin{remark}
Observe that $m_1$ satisfies the assumption of the proposition and more generally any generalized Cauchy distribution $m_\alpha$ with $\alpha \geq 1$. Also, Figure \ref{fig:cauchy} is an illustration of the result of the proposition (with $p_1=(\sqrt{5}-1)/4$, $p_2=\sqrt{2} -1$, $\lambda_0(p)=-1-p+\sqrt{3+2p+p^2}$ and $p_0(\lambda)=(1-\lambda+\frac{\lambda^2}{2})/(2-\lambda)$).

The property $p_0'(1)=1/2$ means that the curve has $\lambda=2p$ as tangent in $(1/2,1)$.


On the other hand, the assumptions on $J$ guarantee that $J'$ is sub-linear, \textit{i.e.}
$J'(a+b)\leq J'(a)+J'(b)$ for all $a,b \in [0,1/2]$,
and that
$t\mapsto\frac{J(t)}{t^{2}}$ is decreasing on $(0,1/2)$.
We will make use of these properties repeatedly.
\end{remark}

\begin{proof}
The proof consists in studying various functions of two variables. Such studies are easy exercises but use various small tricks. For the seek of completeness and in order to help the interesting reader, we give most of the details. We shall use Proposition \ref{prop quant 1} and deal with different cases.

\medskip

\textit{Case $0\leq\lambda\leq p$. } We need to compare the perimeters of $E_{1}$, $E_{2}$ and $E_{4}$. To this aim, consider  the function
$L_{12}(p,\lambda):=P(E_{1})-P(E_{2})=2J(\lambda/4)+2J((2p+\lambda)/4) - J((p+\lambda)/2) - J((p-\lambda)/2)$.
Considering the partial derivative with respect to $\lambda$, and using the sub-linearity of $J'$,
one concludes that $\lambda \mapsto L_{12}(p,\lambda)$ is increasing. Since $L_{12}(p,0)=0$, it finally follows that $P(E_{2}) \leq P(E_{1})$.

Consider now the function  $L_{42}(p,\lambda):=P(E_{4})-P(E_{2})=2J((2p-\lambda)/4)+2J((2-\lambda)/4) - J((p+\lambda)/2) - J((p-\lambda)/2)$. One has
$2\partial_\lambda L_{24}(p,\lambda)=-J'((2p-\lambda)/4)-J'((2-\lambda)/4 - J'((p+\lambda)/2) + J'((p-\lambda)/2) \leq 0$ since $J'$ is non-decreasing and $p-\lambda \leq p+\lambda$ (here and below we use the shorthand notation $\partial_\lambda, \partial_p$ for the partial derivatives with respect to $\lambda,p$).
Therefore, $L_{24}(p,\lambda) \geq L_{24}(p,p) = 2J(p/4)+2J((2-p)/4 - J(p)$. Taking the derivative with respect to $p$ , one immediately sees that $p \mapsto L_{24}(p)$ is non-increasing so that
$L_{24}(p,p) \geq L_{24}(1/4,1/4)=2J(1/8)+2J(3/8)-J(1/2)$. Observing that $t \mapsto J(t)/t^2$ is non-increasing, we have
$J(3/8)/(3/8)^2 \geq J(1/2)/(1/2)^2$, which leads to
\begin{equation}\label{eq:j}
2J(3/8) \geq \frac{9}{8}J(1/2) \geq J(1/2).
\end{equation}
Finally we conclude that
$L_{24}(p,\lambda) \geq 0$, which guarantees that $P(E_{2}) \leq P(E_{4})$.
This completes the picture for $0\leq\lambda\leq p$.

\medskip

\textit{Case $p \leq \lambda \leq 2p$.} We need to compare the perimeters of $E_{1}$, $E_{3}$ and $E_{4}$.  We shall first prove that $P(E_1) \geq P(E_3)$ when $p \leq \lambda \leq 1-p$ and that $P(E_1) \geq P(E_4)$ when $1-p \leq \lambda \leq 2p$. This will reduce the study to the comparison of
the perimeters of $E_3$ and $E_4$ only.

Consider first the function $L_{13}(p,\lambda):=P(E_{1})-P(E_{3})=2J(\lambda/4)+2J((2p+\lambda)/4) - J((p+\lambda)/2) - J((\lambda-p)/2)$ with $p \leq \lambda \leq 1-p$.
Since $\lambda \leq 2p$, it holds
$(\lambda-p)/2\leq \lambda/4$. Hence, $L_{13}(p,\lambda) \geq J(\lambda/4)+2J((2p+\lambda)/4) - J((p+\lambda)/2)$. Using twice the fact that $t \mapsto J(t)/t^2$ is non-increasing, we have
$$
2J\left(\frac{2p+\lambda}{4}\right) \geq \frac{1}{2}\left[\frac{2p+\lambda}{p+\lambda}\right]^2 J\left(\frac{p+\lambda}{2}\right)
\quad \mbox{and} \quad
J\left(\frac{\lambda}{4}\right) \geq \left[\frac{\lambda}{2(p+\lambda)}\right]^2 J\left(\frac{p+\lambda}{2}\right),
$$
so that, after few rearrangements
$$
J\left(\frac{\lambda}{4}\right)+2J\left(\frac{2p+\lambda}{4}\right) - J\left(\frac{p+\lambda}{2}\right)
\geq
\frac{p^2-\frac{\lambda^2}{4}}{(p+\lambda)^2}J\left(\frac{p+\lambda}{2}\right) \geq 0,
$$
since $\lambda \leq 2p$. This implies that $P(E_{1}) \geq P(E_{3})$ (when $p \leq \lambda \leq 1-p$).

Consider now the function
$L_{14}(p,\lambda)=:=P(E_{1})-P(E_{4})=2J(\lambda/4)+2J((2p+\lambda)/4) - 2J((2p-\lambda)/4)-2J((2-\lambda)/4)
$ with $1-p \leq \lambda \leq 2p$. Since $\lambda \mapsto L_{14}(p,\lambda)$ is non-decreasing, we have
$L_{14}(p,\lambda) \geq L_{14}(p,1-p)=2J((1-p)/4)-2J((3p-1)/4)$. Last function is non-increasing (in $p$). Hence,
$L_{14}(p,1-p) \geq L_{14}(1/2,1/2)=0$. This guarantees that, as announced $P(E_1) \geq P(E_4)$
when $1-p \leq \lambda \leq 2p$.

At this point is remains to compare $P(E_3)$ and $P(E_4)$
when $p \leq \lambda \leq 2p$,  considering the function
$L_{43}(p,\lambda):=P(E_{4})-P(E_{3})$. We will distinguish between two sub-cases.

We start by dealing with $p \leq \lambda \leq \min (2p,1-p)$. In that case,
$L_{43}(p,\lambda) =2J((2p-\lambda)/4)+2J((2-\lambda)/4) - J((p+\lambda)/2) - J((\lambda-p)/2)$
 is obviously non-increasing in $\lambda$. In order to deduce the sign of $L_{43}$
we need to study the extreme points $H(p):=L_{43}(p,p)$ and $G(p):=L_{43}(p,\min(2p,1-p))$:
\begin{center}
\begin{variations}
 \lambda &  p          &         &  \min (2p,1-p)  \\ \filet
\m{\lambda \mapsto L_{43}(p,\lambda)} & \h{H(p)} &  \d &   G(p) \;\;\;\;\;\;\;\;\;\;\; \\ \filet
\end{variations}
\end{center}
First, we observe that $H(p):=L_{43}(p,p)=2J(p/4)+2J((2-p)/4) - J(p)$ is non-increasing (take the derivative)
so that $H(p) \geq H(1/2) = 2J(1/8)+2J(3/8)-J(1/2) >0$ thanks to \eqref{eq:j}.
Then, we notice that $p \mapsto G(p):=L_{43}(p,\min(2p,1-p))$ is obviously non-increasing on $[0,1/3]$ and
non-decreasing on $[1/3,1/2]$. Since
$G(0)=L_{43}(0,0)=2J(1/2)>0$, $G(1/2)=L_{43}(1/2,1/2) = 2J(1/8)+2J(3/8)-J(1/2) >0$ thanks to \eqref{eq:j}, and
$G(1/3)=L_{43}(1/3,2/3)=2J(1/3)-J(1/6)-J(1/2)<0$ (since the slope $[J(1/2)-J(1/3)]/(1/6)$ is larger, by convexity of $J$,  than the slope $[J(1/3)-J(1/6)]/(1/6)$), we end up with the following diagram:
\begin{center}
\begin{variations}
  p      &  0          &      & p_1  &        &1/3             &      & p_2  &     &        1/2   \\ \filet
\m{G} & \h{G(0) >0} &  \dh & \m 0 &  \db   &  G(1/3)<0  & \cb  & \m 0 & \ch &  \h{G(1/2) >0}\\ \filet
\end{variations}
\end{center}
for some $p_1 \in (0,1/3)$ and some $p_2 \in (1/3,1/2)$.
From this we conclude that $P(E_{4}) \geq P(E_{3})$ when $p \in [0,p_1] \cup[p_2,1/2]$, and that
$P(E_{4})-P(E_{3})$ changes sign (at a unique point $\lambda_0(p)$) when $\lambda$ varies and $p \in (p_1,p_2)$ is fixed.
This leads to the existence of the function $\lambda_0$.
This completes the picture for $p\leq\lambda\leq \min \{2p,1-p\}$.

Consider finally the range $1-p \leq \lambda \leq 2p$ (which exists only if $p \in [1/3,1/2]$).
In that case, the function $L_{43}(p,\lambda)$ reads
$$
L_{43}(p,\lambda) = 2J((2p-\lambda)/4)+2J((2-\lambda)/4) - J(1-(p+\lambda)/2) - J((\lambda-p)/2) .
$$
Here we used one again the symmetry of $J$ about $1/2$ in order to deal only with variables belonging to $[0,1/2]$ (observe in particular that $(p+\lambda)/2 \geq 1/2$). The map $p \mapsto L_{43}(p,\lambda)$ is clearly increasing.
Hence, we need to study the extremal points $H(\lambda):=L_{43}(\max(\lambda/2,1-\lambda),\lambda)$
and $G(\lambda):=L_{43}(1/2,\lambda)$:
\begin{center}
\begin{variations}
p                               &  \max(\lambda/2,1-\lambda)  &         &  1/2  \\ \filet
\m{p \mapsto L_{43}(p,\lambda)} & \;\;\;\;\;\;\;\;\;\;\; H(\lambda)  & \c &  \h{G(\lambda)} \\ \filet
\end{variations}
\end{center}
Computing $G'$, and using the sub-linearity of $J'$, we conclude that
$G$ is decreasing. Since $G(1)=0$ we are guaranteed that $G(\lambda) > 0$ for any $\lambda \in [1/2,1)$.
On the other hand, note that
$$
H(\lambda)=
\begin{cases}
2J(\frac{2-3\lambda}{4})+2J(\frac{2-\lambda}{4}) - J(\frac 12) - J(\frac{2\lambda-1}{2}) & \mbox{if } \frac{1}{2} \leq \lambda \leq \frac{2}{3} \\
2J(\frac{2-\lambda}{4}) - J(\frac{4-3\lambda}{4}) - J(\frac \lambda 4) & \mbox{if } \frac{2}{3} \leq \lambda \leq 1 .
\end{cases}
$$
Now, in the range $\frac{1}{2} \leq \lambda \leq \frac{2}{3}$, $H$ this obviously decreasing.
While in the range $\frac{2}{3} \leq \lambda \leq 1$, computing the derivative, and using that
$\frac{\lambda}{4} \leq \frac{2-\lambda}{4} \leq \frac{4-3\lambda}{4}$, we conclude that
$H$ is increasing.
Then, observe that $H(1/2)= 2J(1/8)+2J(3/8)-J(1/2) >0$, by \eqref{eq:j}. Also,
$H(2/3) =  2J(1/3)-J(1/2)-J(1/6) = -[J(1/2)-J(1/3)]+[J(1/3)-J(1/6)]<0$ since the slope
$[J(1/2)-J(1/3)]/(1/6)$ is greater, by convexity of $J$,  than the slope
$[J(1/3)-J(1/6)]/(1/6)$ and $H(1)=0$. We end up with the following diagram
\begin{center}
\begin{variations}
   \lambda      &  1/2         & & \lambda_1   &      &  2/3      &     &   1   \\ \filet
\m{H} & \h{H(1/2) >0} &   \dh & \m 0 & \db  &  H(2/3)<0 &   \c &  \h 0\\ \filet
\end{variations}
\end{center}
It follows that $P(E_{4}) \geq P(E_{3})$ when $\lambda \in [1/2,\lambda_1]$, and that
$P(E_{4})-P(E_{3})$ changes sign (at a unique point $p_1(\lambda)$) when $p$ varies and $\lambda \in (\lambda_1,1)$ is fixed.
This leads to the existence of the function $p_0$ and completes the picture in the range
$1-p \leq \lambda \leq 2p$.

It remains to show that $p_0$ is $\mathcal{C}^1$, increasing, $p_0(1-p_2)=p_2$ and that $p_0'(1)=1/2$.
That $p_0(1-p_2)=p_2$ follows from the fact that the perimeter is a continuous function of the variables $p$ and $\lambda$. The remaining properties follow from the implicit equation $L_{43}(p_0(\lambda),\lambda)=0$ and the implicit function theorem.
This ends the proof.
\end{proof}


\subsection{Estimates on the deficit}

In this section we prove a quantitative estimate on the deficit.

The deficit of a set $E$ is  defined as
\begin{equation}
\delta(E)=\left\{
\begin{array}
[c]{lll}%
P(E)-P\left( (  -\infty,-\beta_{p} )   \cup ( \beta
_{p},+\infty)\right)   & \text{if} & \mu\left(  E\right)  \leq\frac{1}{2}\\
P(E)-P\left( ( -\alpha_{p},\alpha_{p})\right)   & \text{if} &
\mu\left(  E\right)  \geq\frac{1}{2}.
\end{array}
\right.  \label{deficit}%
\end{equation}
In words, the deficit measures (in the sense of the perimeter) how far the set is from the optimal set in the isoperimetric inequality.

Recall that the convex function $J\colon (0,1/2) \to [0,\infty)$ is said to satisfy the $\nabla_2$-condition if there exists $\varepsilon >0$ such that, for all $x \in (0,1/2)$, it holds
$J(x) \geq (2+\varepsilon)J(x/2)$ (see \cite{rao-ren}).

What follows is one of our main theorems.
\begin{theorem}\label{th asintotico}
Let $\mu \in \mathcal{F}$ and assume that $J\in C^{2}\left(0,\frac{1}{2}\right)$.
Assume furthermore that $M(p):=\inf_{t \in [p/2,1/2]} J''(t) >0$ for all $p \in (0,1/2]$.
Fix $p \in [0,1/2]$ and $\lambda \in [0,2p]$. Then, there exists $c=c(p)>0$ and $c'>0$
such that the following holds:\\
$(i)$ for any Borel set $E$ of measure $p$ and asymmetry $\lambda$, it holds
\begin{equation} \label{deficit asintotico}
\delta(E) \geq c\left[(1-\lambda)^2+(1-2p)\right]\lambda^{2} ;
\end{equation}
$(ii)$ if in addition $J$ satisfies the $\nabla_2$-condition with $\varepsilon \in (0,1)$, $J'$ is concave on $(0,1/2)$ and $J'(0^+)=0$, then for any Borel set $E \not \ni 0$ of measure $p$ and asymmetry $\lambda$, it holds
\begin{equation} \label{deficit asintotico bis}
\delta(E) \geq c' \lambda^{2} .
\end{equation}
Moreover, one can choose $c'=\varepsilon J''(1/2^-)/32$ and
\begin{equation} \label{label}
c=\frac{1}{32} \min\left(8J'(\frac{p}{2}),M(p),16J'(\frac{1}{6}),8[J(\frac{1}{2})-2J(\frac{1}{4})], 4M\left(\frac{J(1/2)-2J(1/4)}{J'(1/2^-)}\right)\right) .
\end{equation}
\end{theorem}

\begin{remark}
Before giving the proof of Theorem \ref{th asintotico}, we comment on the result.
\begin{itemize}
\item[(a)]
We stress that the constant $c'$ in the right hand side of \eqref{deficit asintotico bis} does not depend on $p$. This part of the theorem will be useful in the next section.
Moreover, the quantity $J(\frac{1}{2})-2J(\frac{1}{4})$ in \eqref{label} is positive thanks to Lemma \ref{lemino}.

\item[(b)]
Using Lemma \ref{sera}, the above result extends at once to the whole region $p \in [0,1]$:
given $p\in[0,1]$ and $\lambda \in [0,2\min(p,1-p)]$, there  exists positive constant $c''=c''(p)$
such that for any Borel set $E$ of measure $p$ and asymmetry $\lambda$, it holds
\begin{equation*}
\delta(E)\geq c''\left[(1-\lambda)^2+(1-2\min(p,1-p))\right]\lambda^{2} .
\end{equation*}

\item[(c)]
The assumptions $J\in C^{2}\left(0,\frac{1}{2}\right)$ and $M(p):=\inf_{t \in [p/2,1/2]} J''(t) >0$ for all $p \in (0,1/2]$ are technical. The result would certainly hold under weaker assumptions.
Also, the constant $c$ is clearly not optimal.

\item[(d)]
It is possible to construct an example for which $\lambda$ is close to $1$
and $\delta$ is small. More precisely, given $\varepsilon,\eta \in (0,1/2)$, consider the set
defined in \eqref{E3} with $p=\frac{1}{2}-\eta$ and $\lambda=1-\varepsilon$: $E=E(\varepsilon)=\left(-\beta_{\frac{1}{2}+\eta-\varepsilon},-\beta_{\frac{1}{2}+\eta+\varepsilon} \right)$. Assuming that $J$ is twice differentiable, an expansion
for $\varepsilon, \eta$ small (recall that $J$ is symmetric about $1/2$) leads to
\begin{align*}
\delta(E)
& =
P(E) - P\left( ( -\infty,-\beta_{p})  \cup(  \beta_{p},+\infty )\right) \\
& =
J\left(\frac{1}{4}+ \frac{1}{2}(\eta+\varepsilon)\right)
+ J\left(\frac{1}{4}+  \frac{1}{2}(\eta-\varepsilon)\right) - 2 J\left(\frac{1}{4}-\frac{1}{2}\eta \right) \\
& =
2J' \left(\frac{1}{4}\right) \eta + \frac{1}{4}J''\left(\frac{1}{4}\right) \varepsilon^2
+o(\eta^2)+o(\varepsilon^2)
\approx
(1-\lambda)^2+(1-2p)
\end{align*}
where $a \approx  b$ means that $\frac{a}{c} \leq b \leq ca$ for some constant $c$.
Hence, for $\varepsilon$ and $\eta$ small, the deficit $\delta(E)$ is small while the asymmetry $\lambda(E)$ is close to 1.
This anomalous phenomenon is coming from the fact that, at $p=1/2$, extremal sets have two different shapes.
Indeed, in the above example, $E$ is closer to the set $(-\alpha_{p},\alpha_{p})$ than to the isoperimetric set $(-\infty,-\beta_{p})\cup(\beta_{p},+\infty)$!

\item[(e)]
Observe that, according to the above example,
the pre-factor $(1-\lambda)^2 + (1-2p)$, in \eqref{deficit asintotico}, is necessary.
Finally, we stress that the behaviour $\lambda^{2}$ in \eqref{deficit asintotico} and in \eqref{deficit asintotico bis} is optimal.
\end{itemize}
\end{remark}

\begin{proof}[Proof of Theorem \ref{th asintotico}]
Fix $p \in [0,1/2]$ and $\lambda \in [0,2p]$ and
a Borel set $E$ of measure $p$ and asymmetry $\lambda$.
We start by proving Point $(i)$.

By Proposition \ref{prop quant 1}, we actually only need to prove that
$$
\delta(E_i) \geq c\left[(1-\lambda)^2+(1-2p)\right]\lambda^{2}
$$
for $E_i$, $i=1,2,3,4$, defined in \eqref{E1}, ..., \eqref{E4}.

We shall deal with each one of these sets and with the different ranges separately.
Also we shall use repeatedly, without any further mention, that $\lambda \in [0,1]$ so that $1 \geq \lambda \geq \lambda^2$
and that $1 \geq \frac{1}{2}[(1-2p)+(1-\lambda)^2]$.

$\bullet$ We start by dealing with the set $E_1$ and $0 \leq \lambda \leq 2p$.
Since $J'$ is non-decreasing, we have
\begin{align} \label{eq:E1}
\delta(E_1)
& =
P(E_1) - P\left((  -\infty,-\beta_{p})  \cup(  \beta
_{p},+\infty)\right)
=
2 \left( J\left(\frac{p}{2}+\frac{\lambda}{4}\right) + J\left(\frac{\lambda}{4}\right) - J\left(\frac{p}{2}\right)\right) \nonumber  \\
& \geq
2 \int_\frac{p}{2}^{\frac{p}{2}+\frac{\lambda}{4}} J'(t)dt
\geq
\frac{J'(\frac{p}{2})}{2}  \lambda
\geq
\frac{J'(\frac{p}{2})}{4} [(1-2p)+(1-\lambda)^2] \lambda^2.
\end{align}

$\bullet$ Consider now the set $E_2$ with $0 \leq \lambda \leq p$. One has
\begin{align} \label{eq:E2}
\delta(E_2)
& =
P(E_2) - P\left( ( -\infty,-\beta_{p})  \cup(  \beta_{p},+\infty) \right)
 =
J\left(\frac{p}{2}+\frac{\lambda}{2}\right) + J\left(\frac{p}{2}-\frac{\lambda}{2}\right) - 2 J(\frac{p}{2}) \nonumber  \\
& =
\int_\frac{p}{2}^{\frac{p}{2}+\frac{\lambda}{2}} \int_{t-\frac{\lambda}{2}}^t J''(u)du dt
\geq
\int_{\frac{p}{2}+\frac{\lambda}{4}}^{\frac{p}{2}+\frac{\lambda}{2}} \int_{t-\frac{\lambda}{4}}^t J''(u)du dt \geq
M(p) \left( \frac{\lambda}{4} \right)^2
\nonumber \\
&
\geq \frac{M(p)}{32} [(1-2p)+(1-\lambda)^2] \lambda^2  .
\end{align}

$\bullet$  We turn to the set $E_3$ with $p \leq \lambda \leq 2p$. We need to distinguish between
two different cases, namely
$p \leq \lambda \leq min(1-p,2p)$ and $min(1-p,2p) \leq \lambda \leq 2p$ (which holds only if $p\geq 1/3$).
For $p \leq \lambda \leq min(1-p,2p)$, by monotonicity (take $\lambda=p$), we have
\begin{align}
\delta(E_3)
& =
P(E_3) - P\left( ( -\infty,-\beta_{p})  \cup(  \beta
_{p},+\infty )\right)  =
J\left(\frac{p}{2}+\frac{\lambda}{2}\right) + J\left(\frac{\lambda}{2}-\frac{p}{2}\right) - 2 J\left(\frac{p}{2}\right) \nonumber  \\
& \geq
J(p)-2J\left(\frac{p}{2}\right) \label{eq:E3bis} \\
& \geq
\left[J(p)-2J\left(\frac{p}{2}\right)\right] [(1-2p)+(1-\lambda)^2] \lambda^2 . \label{eq:E3}
\end{align}

When $min(1-p,2p) \leq \lambda \leq 2p$ (which implies $p\geq 1/3$), it holds $(\lambda + p)/2 \geq 1/2$. Hence, by symmetry of $J$ about $1/2$ (and in order to deal only with increments belonging to $[0,1/2]$), we have $\delta(E_3) = J\left(1-\frac{p}{2}-\frac{\lambda}{2}\right) + J\left(\frac{\lambda}{2}-\frac{p}{2}\right) - 2 J\left(\frac{p}{2}\right)$.
Using that $J'$ is non-decreasing, we have,
\begin{align} \label{eq:E3ter}
\delta(E_3)
& =
\int_\frac{p}{2}^{1-\frac{\lambda+p}{2}} J'(t) dt - \int_{\frac{\lambda-p}{2}}^\frac{p}{2} J'(t)dt
\nonumber \\
&\geq
J'\left(\frac{p}{2}\right)\left(1-p-\frac{\lambda}{2}\right) - J'\left(\frac{p}{2}\right)\left(p-\frac{\lambda}{2}\right)
=
J'\left(\frac{p}{2}\right)(1-2p) \geq J'\left(\frac{1}{6}\right)(1-2p).
\end{align}
On the other hand, since $p \mapsto \delta(E_3) = J\left(1-\frac{p}{2}-\frac{\lambda}{2}\right) + J\left(\frac{\lambda}{2}-\frac{p}{2}\right) - 2 J\left(\frac{p}{2}\right)$ is non-increasing, we get (take $p=1/2$) by convexity of $J$ on $[0,1/2]$, and using that $J\left(\frac{\lambda}{2}-\frac{1}{4}\right) \geq 0$,
\begin{align*} 
\delta(E_3)
& \geq
J\left(\frac{3}{4}-\frac{\lambda}{2}\right) + J\left(\frac{\lambda}{2}-\frac{1}{4}\right) - 2 J\left(\frac{1}{4}\right)
 \geq
J\left(\frac{1}{2}\right) -\frac{2\lambda -1}{4}J'\left((\frac{1}{2})^-\right)
 - 2 J\left(\frac{1}{4}\right)
\end{align*}
Now by Lemma \ref{lemino}, $c:=J(1/2)-2J(1/4)>0$ so that, for
$\lambda \leq \frac{1}{2} + \frac{c}{J'(1/2^-)}$\footnote{Observe that, by convexity,
$c \leq J(1/2)-J(1/4) \leq J'(1/2^-)/4$ so that  $\frac{1}{2} + \frac{c}{J'(1/2^-)} \in [1/2,3/4]$.}
we get
$\delta(E_3) \geq \frac{c}{2}$, while for $\lambda \geq \frac{1}{2} + \frac{c}{J'(1/2^-)}$ (condition that might be empty), we have
\begin{align*}
\delta(E_3)
& \geq
J\left(\frac{3}{4}-\frac{\lambda}{2}\right) + J\left(\frac{\lambda}{2}-\frac{1}{4}\right) - 2 J\left(\frac{1}{4}\right)
 =
\int_\frac{1}{4}^{\frac{3}{4}-\frac{\lambda}{2}} \int_{t+\frac{\lambda-1}{2}}^t J''(u) du dt \nonumber \\
&\geq
M\left(\lambda -\frac{1}{2}\right) \left( \frac{1-\lambda }{2} \right)^2
\geq
\frac{1}{4}M\left(\frac{c}{J'(1/2^-)}\right) \left(1-\lambda \right)^2 .
\end{align*}
Combining these results, we finally get, in the regime $min(1-p,2p) \leq \lambda \leq 2p$,
\begin{equation} \label{eq:E3quarter}
\delta(E_3) \geq \frac{1}{8} \min(4J'(1/6),2c,M\left(\frac{c}{J'(1/2)}\right) )
[(1-2p)+(1-\lambda)^2] \lambda^2 .
\end{equation}

$\bullet$ Finally we deal with $E_4$. Using that $J'$ is non-decreasing, we have,

\begin{align} \label{eq:E4}
\delta(E_4)
& =
P(E_4) - P\left( ( -\infty,-\beta_{p})  \cup(  \beta
_{p},+\infty)\right)  \nonumber  \\
& =
2J\left(\frac{p}{2}-\frac{\lambda}{4}\right) + 2J\left(\frac{1}{2}-\frac{\lambda}{4}\right) - 2 J\left(\frac{p}{2}\right)
 \geq
2 \left(J\left(\frac{1}{2}-\frac{\lambda}{4}\right) -  J\left(\frac{p}{2}\right) \right) \nonumber  \\
& =
2 \int_\frac{p}{2}^{\frac{1}{2}-\frac{\lambda}{4}} J'(t)dt
 \geq
2J'\left(\frac{p}{2}\right)\left(\frac{1}{2}-\frac{\lambda}{4}-\frac{p}{2}\right)
 =
\frac{J'(\frac{p}{2})}{2}[(1-\lambda)+(1-2p)] \nonumber \\
& \geq
\frac{J'(\frac{p}{2})}{2}[(1-\lambda)^2+(1-2p)] \lambda^2 .
\end{align}

The expected result of Point $(i)$ follows collecting \eqref{eq:E1}, \eqref{eq:E2}, \eqref{eq:E3quarter} and \eqref{eq:E4}. 

Now we turn to the proof of Point $(ii)$. By Proposition \ref{prop quant}, we only need to prove that $\delta(E_i) \geq c'\lambda(E)^{2}$
for $E_i$, $i=1,2,3$.
As for Point $(i)$, we shall deal with each one of these sets in the appropriate ranges given in (\ref{range}) (observe that such ranges may differ from Point $(i)$).

By monotonicity of $J'$, we observe that, for $1-p \leq \lambda \leq 2p$ (which guarantees that $p \geq 1/3$), \eqref{eq:E1} implies
$\delta(E_1) \geq J'(1/6)\lambda^2 /2$. On the other hand, back to the computation leading to Equation \eqref{eq:E2}, we have, by monotonicity of $J''$,
$$
\delta(E_2) = \int_\frac{p}{2}^{\frac{p}{2}+\frac{\lambda}{2}} \int_{t-\frac{\lambda}{2}}^t J''(u)du dt
\geq
J''\left(\frac{p}{2}+\frac{\lambda}{2}\right) \times \left(\frac{\lambda}{2}\right)^2
\geq
\frac{J''(1/2^-)}{4} \lambda^2,
$$
where in the last inequality we used that $p+\lambda \leq 1$. Finally, thanks to \eqref{eq:E3bis}, the $\nabla_2$-condition and the fact that $t \mapsto J(t)/t^2$ is non-increasing (a consequence of  the assumption $J'$ concave), we have
\begin{align*}
\delta(E_3)
& \geq
J(p)-2J(p/2)
\geq
\varepsilon J(p/2)
\geq
\varepsilon J(1/2) p^2 \\
& \geq
\varepsilon J(1/2)
\begin{cases}
\frac{\lambda^2}{4} & \text{if } p \leq 1/3\text{ and } p \leq \lambda \leq 2p  \\
\frac{\lambda^2}{9} & \text{if } p \geq 1/3 \text{ and } p \leq \lambda \leq 1-p
\end{cases}
\geq
\frac{\varepsilon J(1/2)}{4} \lambda^2,
\end{align*}
where in the last line we used that $\lambda \leq 2/3$ (in the range $p \leq \lambda \leq 1-p$).
Collecting the previous computations leads to
$$
\min_{i=1,2,3} \delta(E_i) \geq \varepsilon \min\left( \frac{J'(1/6)}{2},\frac{J''(1/2^-)}{4},\frac{J(1/2)}{4}\right) \lambda^2 .
$$
Using that $x J'(x) \geq J(x)$ (a consequence of the fact that $J(0)=0$ and that $J$ is convex), and that $t \mapsto J'(t)/t$ is non-increasing (since $J'$ is concave and $J'(0^+)=0$), we have
$\frac{1}{2}J'(1/6) \geq \frac{1}{6}J'(1/2) \geq \frac{1}{3}J(1/2)$. Hence,
$\min\left( \frac{J'(1/6)}{2},\frac{J''(1/2^-)}{4},\frac{J(1/2)}{4}\right) = \min\left(\frac{J''(1/2^-)}{4},\frac{J(1/2)}{4}\right)$. Then, since $J'$ is concave and $J'(0^+)=0$, $x J''(x) \leq J'(x)$, $x \in (0,1/2)$.
Also, since $t \mapsto J(t)/t^2$ is non-increasing, we have $ xJ'(x) \leq 2J(x)$, $x \in (0,1/2)$. In turn,
$J''(x) \leq 2J(x)/x^2$ so that $J(1/2) \geq J''(1/2^-)/8$. As a conclusion,
$\min\left(\frac{J''(1/2^-)}{4},\frac{J(1/2)}{3}\right) \geq J''(1/2^-)/32$.

This ends the proof of the theorem.
\end{proof}


\section{Functional forms.} \label{sec:functional}

As it is well known, isoperimetric inequalities have often equivalent functional forms,
see \textit{e.g.}\ \cite{mazja,cheeger}. In this section, using the results of the preceding sections, we shall derive some weak embedding properties, and also some weak Cheeger inequalities, in quantitative forms.

We need first to introduce some notations, and in particular the notion of rearrangement of a function with respect to a probability measure. We refer to \cite{bennett} for more on this topic.

Let $\Omega$ be a measurable set and $u:\Omega\subset \mathbb{R}
\rightarrow \mathbb{R}_+$
be a measurable function. The \emph{level sets} of $u$ are the sets
$$
E_h^u := \left\{ x \in \Omega \colon u(x)>h \right\}, \qquad h \in \mathbb{R}_+ .
$$
Then,  we define the \emph{distribution function} of $u$ as
$\mu_{u}(h)=\mu\left(  E_{h}^{u}\right)$  for every $h\geq0$.
The mapping $h \mapsto \mu_u(h)$ is non-negative, decreasing and right continuous on $\left[
0,+\infty\right[$. Moreover $\mu_{u}$ has a jump at $h$ if and only if $\mu_{u}(\left\{  x\in\Omega:u(x)=h\right\})\neq0$. The \emph{decreasing rearrangement} $u^*$ of $u$ is the generalized inverse of $\mu_{u}(h)$, namely
\[
u^{\ast}(s)=\sup\left\{  h\geq0:\mu_{u}(h)>s\right\}  \text{ \ for }%
s\in\left(  0,\mu(\Omega)\right)  .
\]
Now, let $E$ be a measurable subset of $\mathbb{R}$. We denote by $E^{\#}$ the complement of the (unique) symmetric interval such that
$\mu(E^{\#})=\mu(E)$, that is $E^{\#}=\left(  -\infty,F^{-1}\left(\frac
{\mu(E)}{2}\right)\right)  \cup\left(  -F^{-1}\left(\frac{\mu(E)}{2}\right),+\infty
\right)$. Finally, the \emph{rearrangement} of the function $u$ with respect to $\mu$ is the
function%
\[
u^{\#}(x)=u^{\#}(-x)=u^{\ast}(2F(x))\text{ \ for }-\infty<x<F
^{-1}\left(  \frac{\mu(\Omega)}{2}\right)
\]
defined on $\Omega^{\#}$ (symmetric with respect to
the origin). The rearrangement is non-increasing on $(-\infty,0]$ and non-decreasing on $[0,\infty)$ with $u^\#(0)=0$ and for every $h \geq 0$
$$
E_{h}^{u^{\#}}=\left(  E_{h}^{u}\right)  ^{\#}=\left(
-\infty,F^{-1}\left(\frac{\mu(E_{h}^{u^{\#}})}{2}\right)\right)  \cup\left(
-F^{-1}\left(\frac{\mu(E_{h}^{u^{\#}})}{2}\right),+\infty\right)  .
$$
The idea behind such a construction is that the level sets of $u^\#$ are precisely the extremal sets in the isoperimetric inequality related to $\mu$.
Hence, in view of Theorem \ref{th generale}, when $\mu \in \mathcal{F}$, the definition of the rearrangement above will be useful only for functions $u$ satisfying $\mu\left(  \text{supp}u\right)  \leq \frac{1}{2}$. Indeed, for function with $\mu\left(  \text{supp}u\right)  > \frac{1}{2}$,
one would have to consider on one hand sets $E^\#$ that are complement of symmetric intervals
(for sets of measure less than $1/2$), and on the other hand symmetric intervals for sets of measure greater than $1/2$, which is impossible: there is no construction of rearrangement compatible with those two shapes.
Observe that by construction one can easily check that the rearrangement is an homogeneous mapping. More precisely, if $v=\lambda u$, with $\lambda >0$, then, $v^\#=\lambda u^\#$.

Finally, recall that $m$ is a $\mu$-median of $u$ if $\mu(\{u \geq m\} \geq 1/2$ and $\mu(\{u \leq m\} \geq 1/2$.

We are now in position to state our embedding results.

\subsection{Embedding inequality}

The following result can be obtained in a classical way
(see  \textit{e.g.} \cite[Corollary 8.2]{bobkov-k}, \cite{martin}).

\begin{proposition}
Let $\mu$ be a probability measure, on the line, satisfying $P(E) \geq I(\mu(E))$ for all Borel set $E \subset \mathbb{R}$, for some isoperimetric profile $I$.
Then, for any non-negative smooth function
$u$ on $\mathbb{R}$ with $\mu$-median zero, it holds
\begin{equation}\label{iso fun}
\sup_{h \geq 0} \left\{ h I(\mu( u > h))\right\} =
\underset{0<t<1/2}{\sup}\left\{  u^{\ast}(t)\text{ }I\left(  t\right) \right\}
\leq
\int_{\mathbb{R}}\left\vert u^{\prime}\right\vert d\mu.
\end{equation}
\end{proposition}

We shall use Theorem \ref{th generale} to give some explicit examples of application
in the setting of log-convex probability measures.

Recall that, for any measurable function $u$ (on $\mathbb{R}$) and any $p \geq 1$, the weighted Lorentz pseudo-norm  is defined by
$$
\| u \|_{\mathbb{L}^{p,\infty}(\mu)} : = \sup_{t >0} \left\{ t \mu(|u|>t)^{1/p}\right\}  .
$$
Now, since the two-sided exponential measure $\mu_{1}$ satisfies the isoperimetric inequality with isoperimetric profile $I(t)=\min(t,1-t)$ (see Remark \ref{explicit}),  Inequality (\ref{iso fun}) implies
\begin{subequations}
\[
\left\Vert u\right\Vert _{L^{1,\infty}(\mu_{1})}
\leq
\int_{\mathbb{R}}\left\vert u^{\prime}\right\vert d\mu_{1}
\]
for any smooth positive function $u$ with $\mu_{1}$-median zero.

For the Cauchy distribution (\ref{mis 1}) with parameter $\alpha >0$, whose isoperimetric profile is $I(t)=\alpha \min(t,1-t)^{1+\frac{1}{\alpha}}$ (see Remark \ref{explicit}), we have (see also
\cite{bobkov-k})
\end{subequations}
\[
\left\Vert u\right\Vert _{L^{\frac{\alpha}{\alpha+1},\infty}(m_{\alpha})}
\leq
\frac{1}{\alpha} \int_{\mathbb{R}}\left\vert u^{\prime}\right\vert dm_{\alpha}
\]
for any smooth positive function $u$ with $m_{\alpha}$-median zero.

Finally the probability measure $\mu_\Phi$ defined in (\ref{mis mu fi}) with
$\Phi(x)=|x|^{\alpha}$ and $0<\alpha<1$ has an
isoperimetric profile comparable to $\min(t,1-t) \left( \log \frac{1}{\min(t,1-t)}\right)^{\frac{\alpha-1}{\alpha}}$ (see  \cite[Proposition 3.22]{roberto}).
Hence, Inequality (\ref{iso fun}) implies
\[
\left\Vert u\right\Vert _{L^{1,\infty}\left(  \log L\right)^{1-\frac
{1}{\alpha}}(\mu_{\Phi})}
\leq
c\int_{\mathbb{R}
}\left\vert u^{\prime}\right\vert d\mu_{\Phi}
\]
for some constant $c>0$ and any smooth positive function $u$ on $\mathbb{R}$
with $\mu_{\Phi}$-median zero (we refer to \cite{bennett} for a definition of the Orlicz-Lorentz spaces $L^{p,\infty}\left(\log L\right)^{\beta}(\mu)$ for $p>1$ and $\beta\in\mathbb{R}$).

\subsection{Weak inequality of Cheeger type}

In this subsection we investigate on the consequences of the isoperimetric inequality and the quantitative isoperimetric inequality in terms of Cheeger type inequalities.
Indeed, as observed by Bobkov \cite{bobkov-k}, isoperimetric inequalities are equivalent to
some weak Cheeger inequalities we present now.

A probability measure $\mu$, on the line, is said to
satisfy a \emph{weak Cheeger inequality} if there exists some non-increasing function
$\beta:(0,\infty)\rightarrow\lbrack0,\infty)$ such that for every smooth
$u \colon \mathbb{R} \to \mathbb{R}$ with $\mu$-median zero, it holds
\[
\int_{\mathbb{R}}\left\vert u\right\vert d\mu
\leq
\beta(s)\int_{\mathbb{R}}\left\vert u^{\prime}\right\vert d\mu+s\text{ Osc}(u)
\qquad
\forall s>0,
\]
where $\mathrm{Osc}(u) := \sup(u) - \inf(u)$ is the oscillation of $u$.
Since $\int_{\mathbb{R}}\left\vert u\right\vert d\mu\leq\frac{1}{2} \mathrm{Osc}(u)$, observe that only
the values $s\in(0,1/2]$ are relevant in the above definition. Moreover, without loss of generality we can assume that $\beta(s)>0$ for all $s \in (0,1/2)$.
Such inequalities were introduced by Bobkov \cite{bobkov-k}, inspired by a notion of weak Poincar\'e inequality (an $L^2$ analogue) due to Rockner and Wang \cite{rockner-wang},
 and further analysis have been done in \cite{roberto,gozlan-r-s}.

The relationship between $\beta(s)$ and the isoperimetric profile is explained in the following proposition (that holds in more general situations).

\begin{proposition}[Bobkov \cite{bobkov-k}]\label{prop funz}
Let $\mu \in \mathcal{F}$. There is an equivalence between the following two
statements (where $\widetilde I$ is symmetric around $1/2$)
\begin{itemize}
\item[(1)] for all $s>0$ and all smooth $u$ with $\mu$-median equal to 0,
$$
\int  |u|  d\mu  \leq  \beta(s) \int  \left\vert u^{\prime}\right\vert  d\mu  +  s  \mathrm{Osc}(u) ,
$$
\item[(2)]  for all Borel set $E$ with $0<\mu(E)<1$,
$$
P(E) \geq  \widetilde  I(\mu(E)) ,
$$
\end{itemize}
where $\beta$ and $I$ are related by the duality relation
\begin{equation}\label {beta}
\beta(s)  =  \sup_{s\leq t \leq \frac 12}  \frac{t-s}{\widetilde  I(t)}  ,
 \quad
\widetilde  I(t)  =  \sup_{0<s\leq t}  \frac{t-s}{\beta(s)}  \quad \text{ for } t \leq \frac 12  .
\end{equation}
\end{proposition}
We notice that it is an open problem to find the extremal functions, if any, in the weak Cheeger inequality above, when $\mu$ is a log-convex probability measure, even in simple examples such as the Cauchy distribution.

In the next proposition, we extend the latter to quantitative isoperimetric and quantitative weak Cheeger inequalities. Then we may apply this result to log-convex probability measures.

\begin{proposition} \label{prop quantitativa 1}
Let $\mu \in \mathcal{F}$ and $\Psi,\Psi' \colon \mathbb{R} \to \mathbb{R}$ be two convex functions, non-decreasing on $(0,\infty)$. Then, the following two statements are equivalent:
\begin{itemize}
\item[(1)] there exists a non-increasing function $\beta \colon (0,\infty) \to [0,\infty)$
such that for all $s>0$ and all smooth $u$ with $\mu$-median equal to 0 and such that
 $0\notin \text{supp}u$,
\begin{equation} \label{eq:qc}
\beta(s) \Psi \left( \int_{\mathbb{R}}\left\vert \left\vert u\right\vert -u  ^{\#}\right \vert d\mu(x)\right) +
\int  |u|  d\mu  \leq  \beta(s) \int  \left\vert u^{\prime}\right\vert  d\mu  +  2s  \mathrm{Osc}(u) ,
\end{equation}
\item[(2)]  there exists a function $\widetilde  I$ symmetric around $1/2$ such that for all Borel set $E \notni 0$ with $0<\mu(E)<1$,
\begin{equation} \label{eq:qi}
P(E) \geq  \widetilde  I(\mu(E)) + \Psi'(\lambda(E)),
\end{equation}
\end{itemize}
Moreover, $(1) \Rightarrow (2)$ with $I$ symmetric around $1/2$ and $\widetilde  I(t):= \sup_{0\leq s \leq \frac{t}{2}} \frac{t-2s}{\beta(s)}$ for $t \in (0,1/2)$ and $\Psi'=\Psi$;
and $(2) \Rightarrow (1)$ with $\beta(s):= \sup_{s \leq t \leq \frac{1}{2}} \frac{t-s}{\widetilde  I(t)}$
and $\Psi(\cdot):=2 \Psi'(\frac{\cdot}{2})$.
\end{proposition}

\begin{remark}
Observe that there is not a pure equivalence between the two statements, as in Proposition
\ref{prop funz}. Indeed, there is a loss of a factor 2, which is technical (there would be no loss if $\Psi(x)=|x|$ for all $x$). The restriction $0 \notin \text{supp}u$ is also technical and
is necessary in order to apply Theorem \ref{th asintotico}.
\end{remark}

Before proving Proposition \ref{prop quantitativa 1} let us apply the result to our setting.
Assume that $\mu \in \mathcal{F}$ and that  $J(t)\in C^2\left(0,\frac{1}{2}\right)$, $J'$ is concave on $\left(0,\frac{1}{2}\right)$ with $J'(0^{+})=0$ and $J$ satisfies the $\nabla_2$-condition with $\varepsilon \in (0,1)$ so that the assumptions of Point $(ii)$ of Theorem \ref{th asintotico} are satisfied. Therefore, thanks to the aforementioned theorem, Equation
\eqref{eq:qi} holds with $\widetilde I(x)=2J(x/2)$ (see Remark \ref{explicit}) and $\Psi'(x)=c'x^2$
(with $c'$ a universal constant). Hence by Proposition \ref{prop quantitativa 1}, the quantitative weak cheeger inequality \eqref{eq:qc} holds with the corresponding $\beta$
and $\Psi(x)=c'x^2/2$ (note that in some explicit examples, such as the Cauchy distribution, $\beta$ can be computed explicitly \cite{roberto}).

In order to prove Proposition \ref{prop quantitativa 1} we need the following
technical lemmas whose proofs can be found at the end of the section. The first lemma relates the
symmetric difference of the level sets of $|u|$ and to the level sets of the positive and negative part of $u$.

\begin{lemma} \label{Lemma somma lambda}
Let  $\mu \in \mathcal{F}$ and $u \colon \mathbb{R} \to \mathbb{R}$ be a smooth function
with $\mu$-median zero. Define $u^+:=\max(u,0)$ and $u^-:=\max(-u,0)$. Then
\begin{equation*} \label{h}
\lambda(E_{h}^{u^{+}})+\lambda(E_{h}^{u^{-}})\geq\lambda(E_{h}^{\left\vert
u\right\vert })\qquad \text{ \ for all }h>0.
\end{equation*}
\end{lemma}

The second lemma \cite{fusco-sobolev} bounds from above the $L^1$ distance between two functions in terms of the measure of the symmetric difference of their level sets. Since the proof is short and elementary, we shall give it for completeness.

\begin{lemma}[\cite{fusco-sobolev}] \label{Lemma f-g}
Let $\mu \in \mathbb{F}$ and $\Psi \colon \mathbb{R} \to \mathbb{R}$ be a convex function. Then,
for any non-negative functions $u,v \colon \mathbb{R} \to \mathbb{R}$, bounded by $1$, it holds
$$
\Psi\left(  \int_{\mathbb{R} }\left\vert u-v\right\vert d\mu\right)
\leq
\int_{0}^{1}\Psi\left(  \mu\left( E_{h}^{u}\triangle E_{h}^{v}\right)  \right)  dh .
$$
\end{lemma}

We are now in position to prove Proposition \ref{prop quantitativa 1}

\begin{proof}[Proof of Proposition \ref{prop quantitativa 1}]
We start with the proof of $(1) \Rightarrow (2)$. Fix a borel set $E \notni 0$.
By standard approximation of the indicator function $\mathds{1}_E$ (see \cite{bob-hudre}),
we get from \eqref{eq:qc} that
$$
 \beta(s) \Psi \left( \int_\mathbb{R} |\mathds{1}_E - (\mathds{1}_E)^\# | d\mu \right) + \mu(E)
\leq \beta(s)P(E) + 2s .
$$
Since $(\mathds{1}_E)^\#  = \mathds{1}_{E^\#}$, we have
$\int_\mathbb{R} |\mathds{1}_E - (\mathds{1}_E)^\# | d\mu = \mu (E \Delta E^\#)=\lambda(E)$. Thus,
for all $s>0$,
$$
\Psi \left( \lambda(E) \right) + \frac{\mu(E)-2s}{\beta(s)} \leq P(E)
$$
 which leads to \eqref{eq:qi} thanks to the definition of $\widetilde I$.

 Now we prove that $(2)$ implies $(1)$.
 Let $u$ be a smooth function whose support does not contain $0$ and such that its median is $0$.
 By approximation, we may assume that $u$ is bounded, and without loss of generality that $\text{ Osc}(u)=1$ (by homogeneity). Now set $u^+=\max(u,0)$ and $u^-=\max(-u,0)$.
 By the coarea formula and \eqref{eq:qi} we have
 $$
\int|{u^\pm}'|d\mu = \int_0^1 P(E_h^{u^\pm})dh  \geq \int_0^1 \widetilde I( \mu(E_h^{u^\pm}))dh
+
\int_0^1 \Psi'(\lambda(E_h^{u^\pm}) ) dh .
$$
Hence, adding the two inequalities, using the convexity of $\Psi'$ and Lemma \ref{Lemma somma lambda} and \ref{Lemma f-g} together with the definition of $beta$, we have
\begin{align*}
\int |u'| d\mu
& =
\int|{u^+}'|d\mu + \int|{u^-}'|d\mu \\
& \geq
\int_0^1 \widetilde I( \mu(E_h^{u^+})) + \widetilde I( \mu(E_h^{u^-}))dh
+ \int_0^1 \Psi'( \lambda(E_h^{u^+})) + \Psi'( \lambda(E_h^{u^-}))dh  \\
& \geq
\int_0^1 \frac{\mu(E_h^{u^+})+ \mu(E_h^{u^-})}{\beta(s)} dh - \frac{2s}{\beta(s)}
+ 2 \int_0^1 \Psi'\left( \frac{\lambda(E_h^{u^+}) + \lambda(E_h^{u^-})}{2} \right)dh    \\
& \geq \frac{\int |u|d\mu}{\beta(s)} + 2  \int_0^1 \Psi'\left( \frac{\lambda(E_h^{|u|})}{2} \right)dh - \frac{2s}{\beta(s)}  \\
& \geq \frac{\int |u|d\mu -2s}{\beta(s)} + 2 \Psi'\left( \frac{\int ||u|-u^\#|d\mu }{2} \right)
- \frac{2s}{\beta(s)} .
\end{align*}
Note that, in the first line of the above computation, we used that, since $u$ is smooth,
the set $\{x \in \mathbb{R} : u'(x) \neq 0 \mbox{ and } u(x)=0\}$ is $\mu$-negligible (see \cite{bcr-06}). Multiplying by $\beta(s)$ leads to the expected result since $\text{ Osc}(u)=1$.
This achieves the proof of the proposition.
\end{proof}

It remains to prove Lemma \ref{Lemma somma lambda} and Lemma \ref{Lemma f-g}.

\begin{proof}[Proof of Lemma \ref{Lemma somma lambda}]
Fix $u \colon \mathbb{R} \to \mathbb{R}$ with $\mu$-median zero.
By the very definition of the asymmetry (and since $0$ is median of $u$), we have
$\frac{\lambda(E_{h}^{u^\pm})}{2}
=
\mu\left(  E_{h}^{u^\pm}\right)
-\mu\left(E_{h}^{u^\pm}\cap\left(E_{h}^{u^\pm}\right)^{\#}\right)$ for all $h>0$.
Then, we observe that $E_{h}^{u^{+}}$ and $E_{h}^{u^{-}}$ are disjoint,
$E_{h}^{u^{+}}\cup E_{h}^{u^{-}}=E_{h}^{\left\vert u\right\vert }$ and
$\left(E_{h}^{u^{\pm}}\right)^{\#}\subseteq \left(E_{h}^{\left\vert u\right\vert }\right)^{\#}$.
Hence,
\begin{align*}
\frac{\lambda(E_{h}^{u^{+}})}{2}+\frac{\lambda(E_{h}^{u^{-}})}{2}
& \geq
\mu\left(  E_{h}^{u^{+}}\right)  +\mu\left(  E_{h}^{u^{-}}\right)  -\left[
\mu\left(  E_{h}^{u^{+}}\cap\left(  E_{h}^{\left\vert u\right\vert }\right)
^{\#}\right)  +\mu\left(  E_{h}^{u^{-}}\cap\left(  E_{h}^{\left\vert
u\right\vert }\right)  ^{\#}\right)  \right]  \\
& =
\mu(E_{h}^{\left\vert u\right\vert })-\mu(E_{h}^{\left\vert u\right\vert
}\cap(E_{h}^{\left\vert u\right\vert })^{\#})=\frac{\lambda(E_{h}
^{\left\vert u\right\vert })}{2}
\end{align*}
which is the expected result.
\end{proof}

\begin{proof}[Proof of Lemma \ref{Lemma f-g}]
By Jensen's inequality we have
\begin{align*}
\int_{0}^{1}\Psi\left(  \mu\left(  E_{h}^{u}\triangle E_{h}^{v}\right)
\right)  dh
&  =\int_{0}^{1}\Psi\left( \int_{\mathbb{R}}\left\vert \chi_{E_{h}^{u}}(x)-\chi_{E_{h}^{v}}(x)\right\vert d\mu(x)\right)   dh\\
&  \geq
\Psi\left(  \int_{0}^{1}\int_{\mathbb{R}}\left\vert \chi_{E_{h}^{u}}(x)-\chi_{E_{h}^{v}}(x)\right\vert d\mu(x)dh\right) \\
&  =
\Psi\left(  \int_{\mathbb{R}}\left\vert \int_{0}^{1}\left[  \chi_{E_{h}^{u}}(x)-\chi_{E_{h}^{v}%
}(x)\right]  dh\right\vert d\mu(x)\right)
\end{align*}
which leads to the expected result.
\end{proof}


\section*{Acknowledgement}
This work was partially done during the visits made by the first two authors to MODAL'X, Universit\'e Paris Ouest Nanterre la D\'efense. Hospitality and support of this
institution is gratefully acknowledged.


\bibliographystyle{plain}
\bibliography{quantitative}

\end{document}